\documentclass[11pt]{imsart}
\setattribute{journal}{name}{}

\usepackage[colorlinks,citecolor=blue,urlcolor=blue]{hyperref}
\usepackage{amsthm,amsmath, a4wide, amssymb, amsthm,color,enumerate,graphicx}
\numberwithin{equation}{section}

\newcommand{\intd}{\,\text{d}}

\newcommand{\argmin}{\mathop{\rm argmin}}%
\newcommand{\interval}{I_n}

\DeclareMathOperator{\cov}{cov}

\DeclareMathOperator{\var}{var}

\DeclareMathOperator{\tr}{tr}

\newtheorem{thm}{Theorem}
\newtheorem{lem}[thm]{Lemma}

\newtheorem{cor}[thm]{Corollary}
\newtheorem{prop}[thm]{Proposition}
\theoremstyle{definition}

\newtheorem{dfn}[thm]{Definition}
\newtheorem{example}[thm]{Example}
\renewcommand{\P}{P}
\newcommand{\sumin}{\sum_{i=1}^n}
\newcommand{\sumjn}{\sum_{j=1}^n}
\newcommand{\ra}{\rightarrow}
\renewcommand{\b}{\beta}
\renewcommand{\k}{\kappa}

\newcommand{\diag}{{\text{diag}}}

\newcommand{\E}{\mathord{\rm E}}

\let\var\Var

\newcommand{\RR}{\mathbb{R}}

\newcommand{\XX}{\mathcal{X}}

\newcommand{\ZZ}{\mathbb{Z}}

\renewcommand{\phi}{\varphi}

\newcommand{\given}{\,|\,}
\renewcommand{\l}{\lambda}

\renewcommand{\th}{\theta}
\newcommand{\e}{\varepsilon}
\newcommand{\X}{\XX}

\renewcommand{\d}{\delta}

\renewcommand{\a}{\alpha}

\begin{document}

\begin{frontmatter}

\title{Adaptive Bayesian credible sets in regression with a Gaussian process prior}
\runtitle{Credible sets with Gaussian process prior}

\author{Suzanne Sniekers\ead[label=e1]{suzanne.sniekers@math.leidenuniv.nl}\thanksref{t1}}
\and 
\author{Aad van der Vaart\ead[label=e2]{avdvaart@math.leidenuniv.nl}\thanksref{t2}}
\thankstext{t1}{Research supported by the Netherlands Organization for Scientific Research (NWO).}
\thankstext{t2}{The research leading to these results has received funding from the European
  Research Council~under~ERC~Grant~Agreement~320637.}
\runauthor{Sniekers et al.}
\affiliation{Leiden University}

\address{Mathematical Institute\\
Leiden University\\
P.O. Box 9512\\
2300 RA Leiden\\
The Netherlands\\
\printead{e1}\\
\printead{e2}}

\begin{abstract}
We investigate two empirical Bayes methods and a hierarchical Bayes method for adapting the scale of a Gaussian process
prior in a nonparametric regression model. We show that all methods lead to a posterior contraction rate
that adapts to the smoothness of the true regression function. Furthermore, we show that
the corresponding credible sets cover the true regression function whenever this function satisfies 
a certain extrapolation condition. This condition depends on the specific method, but is
implied by a condition of self-similarity. The latter condition is shown to be satisfied with probability one under the prior distribution.
\end{abstract}



\end{frontmatter}

\section{Introduction and main result}
We consider the fixed design regression model, 
where we observe a vector $\vec Y_n:=(Y_{1,n},\ldots, Y_{n,n})^T$ with coordinates
\begin{equation}
\label{eq:RegProblem}
Y_{i,n} = f(x_{i,n}) + \e_{i,n}, \qquad i\in \{1,\ldots, n\}.
\end{equation}
Here the parameter $f$ is an unknown  function $f: \X\to\RR$ on some set $\X$,  
the design points $(x_{i,n})$ are 
a known sequence of points in  $\X$, and the (unobservable) errors 
$\e_{i,n}$ are independent standard normal random variables. We are interested
in the performance of a nonparametric Bayesian approach that uses a scaled Gaussian process $\sqrt c W$
as a prior on $f$. We investigate its efficiency to reconstruct the true regression function, and
its ability to quantify the remaining uncertainty in the statistical analysis through the full posterior
distribution. Our main interest is in the dependence of the posterior distribution on the
scaling factor $\sqrt c$ in the Gaussian process, which can be viewed as a bandwidth parameter that can adapt
the prior and posterior distributions to the unknown regularity of the regression function. We consider
empirical and hierarchical Bayes methods to determine this scaling factor,
and study the properties of the resulting plug-in or full posterior distributions.

We denote the prior process for $f=\bigl(f(x): x\in\X\bigr)$ by $W^c=(W^c_x : x\in\X)$,
where $c$ is the scaling factor, and it is assumed that the process $W^c$ is equal in
distribution to the process $\sqrt c\,W^1$. The index set $\X$ may possess a special
structure, but the general results allow it to be arbitrary. These results cover both
one-dimensional and multidimensional domains $\X$.

As a particular example we consider the case that $\X=[0,1]$ and $W^1$ is a standard Brownian
motion. In this case $W^c$ is a mean-zero Gaussian process with covariance function 
$\E W^c_sW^c_t=c\,(s\wedge t)$, and can also be obtained by 
taking a standard Brownian motion on the transformed
time scale $ct$. More generally, for every \emph{self-similar} process $W^1$ of order $\a$ the
process $(\sqrt c\,W^1_t: t\ge 0)$ is equal in distribution to $(W_{tc^{1/(2\a)}}: t\ge0)$ and hence
our present sense of scaling is equivalent to changing the \emph{length scale} of the standard
process. This applies in particular to multifold integrals (indefinite integrals) of Brownian
motion, as considered in \cite{KimeldorfWahba} in connection to spline smoothing.

For a given scale $c$ the Bayesian model is then described by 
\begin{equation}
\label{EqBayesianModel}
\begin{aligned} f&\given c\sim W^c,\\
\vec Y_n&\given f,c\sim \mathcal{N}_n(\vec f_n, I),
\qquad\qquad \vec f_n=\bigl(f(x_{1,n}),\ldots, f(x_{n,n})\bigr)^T.
\end{aligned} 
\end{equation}
The \emph{posterior distribution} given $c$ is by definition the conditional distribution of $f$ given $(\vec Y_n, c)$ in this setup.
As $\vec Y_n$ depends on $f$ only through $\vec f_n$, the conditional distribution of $f$ given $(\vec Y_n,\vec f_n,c)$ 
does not depend on the data $\vec Y_n$ and is the same as the conditional distribution of $f$ given $(\vec f_n,c)$,
which is determined by the prior only. 
Thus we focus on the posterior distribution of $\vec f_n$, which
by standard Gaussian calculus can be seen to satisfy
\begin{equation}
\label{EqPosterior}
\vec f_n\given \vec Y_n,c\sim \mathcal{N}_n\bigl(\hat f_{n,c}, I-\Sigma_{n,c}^{-1}\bigr),\qquad \hat f_{n,c}=(I-\Sigma_{n,c}^{-1})\vec Y_n,
\qquad \Sigma_{n,c} =I+cU_n,
\end{equation}
for $U_n$ the covariance matrix of the unit scale process $W^1$ restricted to the design points $x_{i,n}$.
For instance, for scaled Brownian motion $(U_n)_{i,j}=x_{i,n}\wedge x_{j,n}$.

If $\vec Y_n$ follows the model (\ref{eq:RegProblem}) with a continuous function $f$, then for fixed
$c$ the posterior mean $\hat f_{n,c}$ tends to $\vec{f}_n$ and the posterior covariance matrix
$I-\Sigma_{n,c}^{-1}$ tends to zero as $n\ra\infty$ (see \cite{Cox,vdVvZGaussian}). This remains
true if $c=c_n$ is made dependent on $n$ and allowed to tend to zero or infinity at polynomial rates.  Thus the
posterior distribution given $c=c_n$ contracts to the Dirac measure at $f$ for reasonable $c_n$.  The rate
of contraction depends on $c_n$ and the regularity of the function $f$ jointly. A smaller
value of $c$ corresponds to less variability in the prior process, and yields a posterior distribution
with a less variable mean function and a smaller covariance. This is advantageous if the true regression
function $f$ is fairly regular, but will lead to a suboptimal contraction rate and a too optimistic
quantification of remaining uncertainty in the opposite case (see \cite{vdVvZ,Sniekers}). It is
therefore important to adapt $c$ to the data. We discuss three methods, which turn out to have
similar behaviour, both in terms of contraction rate and uncertainty quantification, although
the sets of functions for which they work differ.

In the \emph{hierarchical Bayes} setup
the parameter $c$ is equipped with a prior, and an ordinary Bayesian analysis is carried out with 
the resulting  mixture of normals prior for $f$. We shall consider the situation that
$c$ follows an inverse Gamma distribution.

In the \emph{empirical Bayes} setup an estimator $\hat c_n$
of the length scale is plugged into the posterior distribution for given $c$. We consider two methods of estimation:
a likelihood-based and a risk-based method. 

The \emph{likelihood-based empirical Bayes} method defines $\hat c_n$ as the maximum
likelihood estimator of $c$ within the \emph{marginal Bayesian model} $\vec Y_n\given c\sim \mathcal{N}(0, \Sigma_{n,c})$,
which follows from (\ref{EqBayesianModel}).
In this marginal model $c$ is the only parameter, and its maximum likelihood estimator is
\begin{equation}
\label{EqLikEB}
\hat c_n=\argmin_{c\in \interval} \Bigl[\log \det\Sigma_{n,c}+\vec Y_n^T \Sigma_{n,c}^{-1}\vec Y_n\Bigr].
\end{equation}
The restriction of $c$ to an interval $\interval$ away from the extremes 0 and $\infty$ is convenient. 
Throughout the paper we shall use 
$$I_n=[\log n/n,n^{m-1}],$$
where $m$ is chosen large enough so that the minimax scaling rates for all smoothness levels
are included. (If (\ref{EqEigenvalues}) holds, then it is chosen equal to
the $m$ in this equation.) 
The likelihood-based empirical Bayes procedure ought to be close to the hierarchical Bayes procedure,
as the posterior density for $c$ is proportional to the marginal density of $\vec Y_n$ given $c$
times the prior density by Bayes's rule, and hence ought to concentrate around $\hat c_n$ in (\ref{EqLikEB}). 
Thus the posterior distribution with a likelihood-based empirical Bayes plug-in for the scale parameter
is sometimes viewed a computationally cheaper version of a true Bayesian analysis.

The \emph{risk-based empirical Bayes} method uses an alternative estimator for $c$ 
that tries to minimize the risk of the posterior mean $\hat f_{n,c}$, which is given by
\begin{align}
\label{EqRiskPostMean}
\E_f \bigl\|\hat f_{n,c}-\vec f_n\bigr\|^2
&=\|-\Sigma_{n,c}^{-1}\vec f_n\|^2+\tr\bigl((I-\Sigma_{n,c}^{-1})^2\bigr).
\end{align}
The first term on the right depends on the unknown function $f$, and hence cannot be used
in a criterion to estimate $c$. An obvious estimate for this term is $\|-\Sigma_{n,c}^{-1}\vec Y_n\|^2$, but it is
biased, as
$$\E_f \|-\Sigma_{n,c}^{-1}\vec Y_n\|^2=\|\Sigma_{n,c}^{-1}\vec f_n\|^2+\E_f \|\Sigma_{n,c}^{-1}\vec \e_n\|^2
=\|\Sigma_{n,c}^{-1}\vec f_n\|^2+\tr(\Sigma_{n,c}^{-2}).$$
This motivates the estimator for $c$ given by 
\begin{align}
\label{EqRIskEB}
\hat c_n &=\argmin_{c\in \interval}\Bigl[\tr\bigl( (I-\Sigma_{n,c}^{-1})^{2}\bigr)-\tr(\Sigma_{n,c}^{-2})+\vec Y_n^T\Sigma_{n,c}^{-2}\vec Y_n\Bigr].
\end{align}
In the special case that $W^c$ is an $(m-1)$-fold integral of Brownian motion,
this estimator was introduced in the context of regression by spline-smoothing. The posterior mean in our setup
is then equal to a penalized least squares estimator for the penalty $\l\int f^{(m)}(x)^2\,dx$, with smoothing parameter
$\l$ equal to $1/(cn)$. See \cite{Wahba,Cox}. 

In Bayesian inference the posterior distribution is used both to reconstruct the regression function $f$, typically
by the posterior mean, and to quantify the uncertainty in this construction, using the spread
of the posterior distribution. In this paper we are interested in the accuracy of these procedures within the
so-called frequentist setup, which assumes that the data  $\vec Y_n$ are generated according to 
model (\ref{eq:RegProblem}) for a given ``true function'' $f$. 
The accuracy of the posterior mean as a point estimator of $f$ can be measured by its risk function
or the contraction rate of the full posterior distribution (see \cite{GGvdV}), as usual.
The accuracy of the uncertainty quantification can be
studied through the coverage and size of \emph{credible sets}, which are data-dependent sets of prescribed posterior
probability. In connection to the empirical Bayes methods we shall first study credible sets of the form
\begin{align} 
\label{EqCredibleSetEB}
\hat C_{n,\eta,M}=\bigl\{f: \|\vec f_n-\hat f_{n,\hat c_n}\|< M r_n(\hat c_n,\eta)\bigr\},
\end{align}
with $\|\cdot\|$  the Euclidean norm. Here $r_n(c,\eta)$ is determined, for given $\eta\in(0,1)$, 
such that the ball of radius $r_n(c,\eta)$ centered at the origin receives probability
$\eta$ under the posterior law of $\vec f_n-\hat f_{n,c}$ given a fixed $c$, which by (\ref{EqPosterior}) is the normal law 
$\mathcal{N}_n(0,I-\Sigma_{n,c}^{-1})$.
In the hierarchical Bayes setup we first select a pair of (nontrivial) quantiles $\hat c_{1,n}(\eta_1)<\hat c_{2,n}(\eta_1)$
in the posterior distribution of $c$, the distribution of $c\given \vec Y_n$ in the Bayesian model (\ref{EqBayesianModel})
augmented with a prior on $c$. We then consider as credible sets for $f$:
\begin{align} 
\label{EqCredibleSetHB}
\hat C_{n,\eta,M}=\bigcup_{\hat c_{1,n}(\eta_1)<c<\hat c_{2,n}(\eta_1)}\bigl\{f: \|\vec f_n-\hat f_{n, c}\|< M r_n(c,\eta_2)\bigr\}.
\end{align}
This two-step construction can exploit that the credible sets for fixed $c$ have a simple description through the
radii $r_n(c,\eta)$. An alternative would be a ball around the hierarchical posterior mean $\int \hat f_{n,c}\,\Pi_n(dc\given \vec Y_n)$.

The uncertainty quantification, by either (\ref{EqCredibleSetEB}) or (\ref{EqCredibleSetHB}),
 is deemed accurate if the sets $\hat C_{n,\eta,M}$ cover the true parameter $f$ with high probability,
if the data are generated according to the model (\ref{eq:RegProblem}). 
In particular, the credible sets are \emph{honest confidence sets} at level $\eta$ for a given class of functions $\mathcal{F}$ if
\[
 \inf_{f\in \mathcal{F}} P_f\bigl(f\in\hat{C}_{n,\eta,M}\bigr) \geq \eta.
\]
The number $r_n(c,\eta)$ is the natural radius of the credible set for fixed $c$ at level $\eta$ in the Bayesian framework. 
The additional constant $M$ in the definitions (\ref{EqCredibleSetEB})--(\ref{EqCredibleSetHB}) of the credible sets 
is required because the Bayesian and frequentist notions of coverage are not the same, and $c$ is estimated.

It is well known that the size of an honest confidence set for a given model $\mathcal{F}$ is determined by
``worst case'' members of $\mathcal{F}$ \cite{Low,JudLam,CaiLow,CaiLow2006,RobVaart,GenWas,Hoff}.
For instance, if $\mathcal{F}$ contains a H\"older ball of regularity $\a$,
then the (random) diameter of the confidence set cannot be of smaller order than $\sqrt n\, n^{-\a/(2\a+1)}$, even if the true function is much smoother.
In other words, the size of honest confidence sets cannot \emph{adapt} to the unknown smoothness of the true regression function.
On the other hand, the posterior contraction rate of the hierarchical Bayes method is known to adapt to unknown regularity, 
in that the rate is faster if the true function is smoother. We show below that the empirical Bayes methods adapt in
a similar manner. Since the corresponding credible sets will have diameter of order the contraction rate, it follows that
these sets cannot be honest over a ``full'' set of functions, such as a H\"older ball. Following \cite{GineNickl,Bull,Szabo} 
we lower our expectation and
investigate honesty over a reduced parameter space, with certain ``inconvenient'' true parameters cut out, as follows.

The distribution of the data depends on the function $f$ only through the vector $\vec f_n$.  A convenient way
to describe this vector is through its coordinates relative to the eigenbasis of the covariance
matrix $U_n$. Write $f_{1,n},\ldots, f_{n,n}$ for the coordinates of $\vec f_n$ relative to this basis, 
i.e.\
$$f_{j,n} := {\vec f_n^T  e_{j,n}}, \qquad j \in \{1,\dots,n\},$$
for $e_{1,n},\ldots, e_{n,n}$ the orthonormal eigenbasis of $U_n$. Let $\l_{1,n},\ldots, \l_{n,n}$ be
the corresponding eigenvalues.

\begin{dfn}[Discrete polished tail]
\label{def:polished}
We say that the function $f$, or the corresponding array $(f_{j,n})$, satisfies the \emph{polished tail condition} if there exist 
constants $L$ and $\rho$ such that for all $c>0$ and sufficiently large $n$ it holds that
\begin{equation}
\label{eq:generalizedpolishedtail}
L\sum_{j: \rho\le c\l_{j,n}\le1}f_{j,n}^2\ge \sum_{j: c\l_{j,n}\le 1}f_{j,n}^2.
\end{equation}
\end{dfn}

The condition may be paraphrased as requiring that the ``energy''  of the signal $f$ in the ``large frequencies''
$\{j: 1\le c\l_{j,n}\le \rho\}$ is at least a fraction $L^{-1}$ of the ``energy'' in the ``frequencies'' $\{j: c\l_{j,n}\le 1\}$. 
Perhaps a better name would be ``self-similar'', but this name is already taken in the literature
for a more special property. The following example shows that the condition 
is similar to the polished tail condition introduced in \cite{Szabo} when the eigenvalues decrease polynomially in $j$. 

\begin{example}
[Polynomial eigenvalues]
If $\l_{j,n}\asymp K_n/j^k$, for some constants $K_n$ and $k>0$, then the discrete polished tail condition is equivalent 
to the existence of constants $L$ and $\rho$ such that, for all sufficiently large $m$ (and hence sufficiently large $n$),
\begin{equation}
\label{eq:EqPolishedTail}
\sum_{j=m}^n f_{j,n}^2\le L\sum_{j=m}^{\rho m\wedge n} f_{j,n}^2.
\end{equation}
Indeed, the condition $c\l_{j,n}\le 1$ is equivalent to $j\ge (cK_n)^{1/k}=:J$, whence 
the right side of (\ref{eq:generalizedpolishedtail}) is bounded above by
$\sum_{j\ge J}f_{j,n}^2$, which is bounded above by $L\sum_{J\le j\le J\rho} f_{j,n}^2$
by (\ref{eq:EqPolishedTail}). This is the left side of (\ref{eq:generalizedpolishedtail}), with $\rho^{-k}$ instead of $\rho$.

In \cite{Szabo} a condition similar to (\ref{eq:EqPolishedTail}) is introduced in a continuous time setup.
We comment on the relationship of these conditions in Section~\ref{SectionPolishedTail}.
\end{example}

The main result of this paper is that all three types of credible sets are honest confidence sets 
over polished tail parameters, of diameter that adapts to the smoothness of $f$. We measure smoothness
through the square norms, for $\a>0$,
\begin{equation}
\label{EqDefFnalpha}
\begin{aligned}
\|f\|_{n,\a}^2&=\frac 1n\sumjn j^{2\a}f_{j,n}^2,\\
\|f\|_{n,\a,\infty}^2&=\frac 1n\sup_{1\le j\le n}j^{1+2\a}f_{j,n}^2.
\end{aligned}
\end{equation}
These norms are in terms of the restriction of $f$ to the grid $(x_{j,n})$. We comment on their
relationship to norms on the full function $f$ in Section~\ref{SectionPolishedTail}. (In general the
coefficients $f_{j,n}$ cannot be directly related to an infinite sequence of Fourier coefficients of $f$,
but for many functions the numbers $f_{j,n}/\sqrt n$, which include the scaling factor $\sqrt n$,  is close to the $j^{\text{th}}$ Fourier coefficient.)

In the following theorem we assume that there exist constants $0<\underline\d\le\overline\d<\infty$ and $m\ge 1$ 
such that the eigenvalues $\l_{1,n},\ldots,\l_{n,n}$ of $U_n$ satisfy
\begin{equation}
\label{EqEigenvalues}
\underline\d\, \frac{n}{j^{m}} \leq\lambda_{j,n} \leq \overline\d\,\frac{n}{j^{m}}.
\end{equation}
Since $\vec W_n$ is distributed as $\sumjn \sqrt{\l_{j,n}}Z_{j,n}e_{j,n}$ for i.i.d.\ standard normal random variables $Z_{j,n}$, 
we have $\E \|W\|_{n,\a}^2=n^{-1}\sumjn j^{2\a}\l_{j,n}$. For the eigenvalues (\ref{EqEigenvalues}) this is uniformly bounded if and only if $\a< (m-1)/2$.
Thus these eigenvalues correspond to modelling the regression function a-priori as ``almost $(m-1)/2$-smooth''.

Let $\mathcal{F}_{n,L}$ be the set of all functions that satisfy the discrete polished tail condition
(\ref{eq:EqPolishedTail}) for given $L$ and satisfy $\sumjn f_{j,n}^2\le d n$ for some sufficiently small constant $d$ 
(that may depend on $\underline\d$ and $m$).

\begin{thm}\label{thm:honest}
Assume that (\ref{EqEigenvalues}) holds.
For sufficiently large $M$ and any $\eta>0$ the credible sets (\ref{EqCredibleSetEB}), with $\hat c_n$ given by (\ref{EqLikEB}) or (\ref{EqRIskEB}),
and the credible sets (\ref{EqCredibleSetHB}) satisfy
$$\inf_{f\in\mathcal{F}_{n,L}} P_f( f \in \hat{C}_{n,\eta,M})\ra 1.$$
Furthermore, for any $\a\in (0,m/2)$, the diameter of the credible sets
$\hat{C}_{n,\eta,M}$  relative to the scaled Euclidean norm $\|\cdot\|_{n,0}$
is of the order $O_{P_f}\bigl(n^{-\a/(1+2\a)}\bigr)$, uniformly in $f$ with $\|f\|_{n,\a}\lesssim 1$ or $\|f\|_{n,\a,\infty}\lesssim 1$.
For the risk-based empirical Bayes method this is even true for $\a\in (0,m)$.
\end{thm}

The theorem is a summary of the main results of the paper as valid for all three methods.
More specific results for the individual methods, with relaxations of the polished tail condition
tailored to the specific method, as well as results that do not assume the eigenvalue condition 
(\ref{EqEigenvalues}), are described below. For example, these results cover functions $f$ on a
two-dimensional domain with eigenvalues of the forms (\ref{EqEigenvaluesDouble}) or (\ref{EqEigenvaluesDoubleSobolev}),
as introduced below.

The second and third assertions of the theorem show that the diameter of the credible sets adapts to
the regularity of the true regression function. The restrictions to regularity levels $\a< m/2$ or
$\a<m$ in the likelihood-based and risk-based  methods stem from
the prior, through the rate of decrease (\ref{EqEigenvalues}) of its eigenvalues,
and the method used. The range $(0,m)$
is bigger than could be expected from the existing literature on Gaussian process priors. For instance, $(m/2-1)$-fold integrated
Brownian motion satisfies (\ref{EqEigenvalues}) and has sample paths of regularity $m/2-1/2$. It has
been documented to be an appropriate prior for functions of exactly regularity $m/2-1/2$, and to become appropriate for
functions of regularities $\a\in (0,m/2]$ after appropriate (deterministic) scaling \cite{Sniekers,vdVvZ,Bartek}. 
The latter property is retained under random scaling by likelihood-based empirical Bayes and hierarchical Bayes methods considered
in the present context (although for $\a=m/2$ an extra logarithmic factor may come in; see Example~\ref{ExampleRectangleLB}; 
the definitions of regularity in the various papers are also not directly comparable).
Surprisingly the risk-based method performs better than the likelihood-based methods, in that it enlarges the
good range to $\a\in (0,m)$. This is caused by the closer connection of the risk-based empirical Bayes method to
the diameter of the credible set, yielding a more appropriate scaling factor $\hat c_n$ for minimizing this diameter.

The diameter of the credible sets is linked to the posterior contraction rate. The rates $O_{P_f}(n^{-\a/(1+2\a)})$
are attained irrespective of $f$ satisfying the polished tail condition, the latter condition being important only for
the coverage. 

The credible sets (\ref{EqCredibleSetEB}) and (\ref{EqCredibleSetHB}) are obtained by considering balls in the
space of function values of $f$ at the design points. An alternative are (sets based on) pointwise intervals of the form
\begin{align} 
\label{EqCredibleSetEBInterval}
\hat C_{n,\eta,M}(x)&=\bigl\{ f: |f(x)-\hat f_{n,\hat c_n}(x)| < M r_n(\hat c_n,\eta,x)\}\\
\noalign{or}
\label{EqCredibleSetHBInterval}
\hat C_{n,\eta,M}(x)&=\bigcup_{\hat c_{1,n}(\eta_1)<c<\hat c_{2,n}(\eta_1)}\bigl\{ f: |f(x)-\hat f_{n,c}(x)| < M r_n(c,\eta_2,x)\},
\end{align}
where $\hat f_{n,c}(x)$ denotes the mean of the marginal posterior distribution of $f(x)$ given $c$ and $r_n(c,\eta,x)$ is determined so that 
\[
 P\bigl( |f(x)-\hat{f}_{n,c}(x)|<r_n(c,\eta,x) \mid \vec{Y}_n,c\bigr)=\eta.
\]
Since this marginal posterior distribution of $f(x)$ given $c$ is
normal with mean $\hat f_{n,c}(x)$, these intervals are easily determined. In particular, for a design point $x=x_{i,n}$
the radius $r_n(c,\eta,x)$ is equal to $z_\eta (1-(\Sigma_{n,c}^{-1})_{i,i})^{1/2}$, for $z_\eta$ the $(1+\eta)/2$-quantile of the
standard normal distribution. When used simultaneously for multiple values of $x$, these intervals form
a \emph{credible band}. 

The study of the coverage of such pointwise intervals and bands requires different techniques from those in the present paper,
and appears to be tractable only for concretely specified prior processes. However, the methods developed here
are suitable when measuring coverage in an averaged fashion that focuses on the fraction of the design points at which the intervals
(\ref{EqCredibleSetEBInterval}) or (\ref{EqCredibleSetHBInterval}) cover the true function.
A similar point of view was taken by \cite{Wahba, CaiLowMa}. The following corollary gives such a result for a subset of design points $x_{i,n}$ that are spread evenly relative to the prior process. More precisely, let
\[
 s_n^2(c,x_{i,n}) := \inf _{a\in \RR^n}\Bigl[c\,\E \bigl(W_{x_{i,n}}^1-a^T\vec W_n^1\bigr)^2+\|a\|^2\Bigr]
\]
denote the posterior variance at the design point $x_{i,n}$ and set
\begin{equation}\label{eq:indices}
 J_n := \Bigl\{ i:  s_n^2(c,x_{i,n}) \ge \frac{C}{n} \sum_{j=1}^n  s_n^2(c,x_{j,n}) \Bigr\}
\end{equation}
for some constant $C$ that is independent of $n$. Then the corollary holds when considering the design points in this set.

%The following corollary gives such a result under the assumption that the design points $x_{i,n}$ are spread evenly relative to the prior process, in the sense that there exists a set $J_n$ such that uniformly in $i,j\in J_n$ and $c\in I_n$ it holds that the posterior variances in the design points
%\begin{equation}
%\label{EqUniformDesignPoints}
%\inf _{a\in \RR^n}\Bigl[c\,\E \bigl(W_{x_{i,n}}^1-a^T\vec W_n^1\bigr)^2+\|a\|^2\Bigr]
%\text{ and }
%\inf _{a\in \RR^n}\Bigl[c\,\E \bigl(W_{x_{j,n}}^1-a^T\vec W_n^1\bigr)^2+\|a\|^2\Bigr]
%\end{equation}
%are of the same order.

In Corollary~3.6 of \cite{Sniekers}, we have seen that Brownian motion satisfies this condition for the set of all design points that satisfy $x_{i,n}\ge C /\sqrt{\log n}$.
%The following corollary shows that the uncertainty quantification through the intervals $\hat C_{n,\eta,M}(x_{i,n})$
%is correct at all the design points, except possibly a fraction.

The following corollary shows that the uncertainty quantification through the intervals $\hat C_{n,\eta,M}(x_{i,n})$
is correct at the design points in the set $J_n$ as long this set is large enough, except possibly a fraction.

\begin{cor}
\label{CorollaryIntervals}
Assume that (\ref{EqEigenvalues}) holds and that the set $J_n$ given in (\ref{eq:indices}) satisfies $|J_n|\sim n$. Fix $\gamma\in(0,1)$, $\eta>0$ and let $\hat{c}_n$ be given by (\ref{EqLikEB}) or (\ref{EqRIskEB}). Then for sufficiently large $M$ the credible sets defined in either (\ref{EqCredibleSetEBInterval}) or (\ref{EqCredibleSetHBInterval}) satisfy
$$\inf_{f\in\mathcal{F}_{n,L}} P_f\Bigl( \frac1n\sum_{i\in J_n} 1\bigl\{f \in \hat{C}_{n,\eta,M}(x_{i,n})\bigr\}\ge \gamma \Bigr)\ra 1.$$ 
Furthermore, if for $i\in J_n$ it also holds that $s_n^2(c,x_{i,n}) \le \frac{C'}{n} \sum_{j=1}^n  s_n^2(c,x_{j,n})$ for some $C'>0$, then for any $\a\in (0,m/2)$ the length of the intervals $\hat{C}_{n,\eta,M}(x_{i,n})$  
is of the order $O_{P_f}\bigl(n^{-\a/(1+2\a)}\bigr)$ uniformly in $i \in J_n$, uniformly in $f$ with $\|f\|_{n,\a}\lesssim 1$ or $\|f\|_{n,\a,\infty}\lesssim 1$.
For the risk-based empirical Bayes method this is even true for $\a\in (0,m)$.
\end{cor}
\noindent The proof of this corollary can be found in Section~\ref{SectionProofs}.

The multiplicative constant $n$ in (\ref{EqEigenvalues}) is motivated by comparison with
the continuous time setup. If the covariance function $K(s,t)=\E W_s^1W_t^1$
of the continuous time process $W^1$ has eigenfunctions $e_j$ satisfying 
$$\int K(s,t)e_j(t)\,ds=\l_je_j(s),$$
then for equidistant design points one may expect that 
$$\sumin K_m(x, x_{i,n})e_j(x_{i,n})\approx n\l_j e_j(x).$$
This suggests both that $\l_{j,n}\approx n\l_j$
and that the ``discrete'' eigenvectors $e_{j,n}$ should be close to the eigenfunctions restricted to the design points.
This is a suggestion only, which already makes little sense when counting the
numbers of eigenvalues involved: $n$ versus $\infty$. Nevertheless, for the
Brownian motion prior the correspondence is exact.

\begin{example}
[Brownian motion]
\label{ExampleBrownianMotion}
The Brownian motion prior permits explicit formulas for eigenbasis and eigenvalues, provided the design points
are taken equal to $x_{i,n}=i/(n+1/2)$ for $i\in\{1,\ldots, n\}$, a slight shift from the usual uniform grid.
The formulas are interesting as they allow to make a connection to the Fourier basis (see Section~\ref{SectionPolishedTail}).

The eigenvectors of the covariance matrix $U_n$ of standard Brownian motion, scaled to unit length, are given by,
for $j\in\{1,\ldots, n\}$,
\begin{equation}
\label{EqEigenvectors}
e_{j,n} = \frac{1}{\sqrt{n+1/2}} \bigl(e_j(x_{1,n}),\ldots, e_j(x_{n,n})\bigr)^T,\qquad e_j(x) =\sqrt{2} \sin\bigl[\bigl(j-\tfrac{1}{2}\bigr)\pi x\bigr].
\end{equation}
The functions $e_j$ are an orthonormal basis of $\{f\in L_2[0,1]: f(0)=0\}$, and happen to be eigenfunctions of the covariance kernel of 
continuous Brownian motion. A similar correspondence is valid for Brownian bridge, but we are not
aware of other examples where the continuous and discrete setups match up so closely.

The eigenvalues of $U_n$ are given by 
$$\lambda_{j,n}=\frac{1}{(4n+2)\sin^2 \bigl( (j-1/2)\pi/(2n+1)\bigr)}.$$
As the argument of the sine is in $[0,\pi/2]$, for which  $2x/\pi \leq \sin x \leq  x$,
there exist numbers $(\underline\d,\overline\d)$ such that
\begin{equation}
\label{EqEigenvaluesBM}
 \underline\d \,\frac{n}{j^2} \leq\frac{1}{(4n+2) \pi^2} \Bigl(\frac{2n+1}{j-\frac{1}{2}}\Bigr)^2  
\leq \lambda_{j,n} \leq \frac{1}{16n+8}  \Bigl( \frac{2n+1}{j-\frac{1}{2}}\Bigr)^2 \leq \overline\d \,\frac{n}{j^2},
\end{equation}
where this inequality holds for all $n$ and $j\geq 1$ if we take $(\underline\d,\overline\d) =(\pi^{-2},3)$, 
and for $j>2$ and $n$ sufficiently large if we let $\overline\d=4/10$.

Standard Brownian motion has sample paths of regularity $1/2$, and has
been documented to become an appropriate prior for
functions of regularities $\a\in (0,1)$ after appropriate scaling \cite{Sniekers,vdVvZ,Bartek}. 
We show in the present paper that the good range is enlarged to $\a\in (0,2)$ provided that
the scaling by the risk-based empirical Bayes method is used.
\end{example}

\begin{example}
[Discrete priors]
Although it often helps intuition to model a function $f$ a-priori by a Gaussian process on 
a ``continuous'' space that encompasses the design points, nothing in the preceding setup 
requires this. In fact, we may turn the construction around, by starting with  
an arbitrary orthonormal basis $e_{1,n},\ldots, e_{n,n}$ and eigenvalues $\l_{1,n},\ldots, \l_{n,n}$,
and next define the prior covariance matrix $U_n$ to be the matrix that has this
as its eigenbasis and eigenvalues, that is, its spectral decomposition is 
\begin{equation}
\label{EqSpectralDecompositionU}
U_n=\sumin \l_{i,n}e_{i,n}e_{i,n}^T.
\end{equation}
Given arbitrary points $x_{1,n},\ldots, x_{n,n}$ the vector $\vec f_n$ is then a-priori
modelled by its coefficients $f_{i,n}$ relative to $e_{1,n},\ldots, e_{n,n}$, which are
independent $\mathcal{N}(0, c\l_{i,n})$-variables. 

One particular example is to retain the eigenvectors of Brownian motion, but to change
the corresponding eigenvalues to (\ref{EqEigenvalues}) for a general $m$. The interpretation of
the norms $\|\cdot\|_{n,\a}$ and $\|\cdot\|_{n,\a,\infty}$ would be the same as for Brownian motion
(as discussed in Section~\ref{SectionPolishedTail}),
but the good rates relative to these norms would now be attained for $\a$ up to $m$ (or $m/2$) rather
than $2$ (or $1$). Our theoretical results show only advantages to taking a larger value of $m$, but
one might guess that a deeper analysis could change this picture.
\end{example}

\begin{example}
[Discrete Laplacian]
\label{ExampleDiscreteLaplacian}
The \emph{discrete Laplacian} is a useful tool to construct ``smooth priors'' on a discrete set of
design points. For  a univariate grid it is closely connected to the Brownian motion prior of
Example~\ref{ExampleBrownianMotion}.
For a countable set $\X$ equipped with a neighbourhood relation $\sim$ the Laplacian is the operator
acting on functions $f: \X\to \RR$,  defined by
$$L(f)(x)=\sum_{y: y\sim x}\bigl[f(y)-f(x)\bigr].$$
Small values of $|Lf|$ indicate that $f$ changes little across its neighbourhoods, whence
$L$ can be used to model smoothness relative to the given neighbourhood structure.

Identification of a function $f: \X\to\RR$ with the infinite vector $\bigl(f(x): x\in \X\bigr)$ 
gives an identification of $L$ with an infinite matrix (with $(x,y)^{\text{th}}$ element equal to  $1$ if $y\not=x$ and $y\sim x$;
equal to $-\#\{y\sim x\}$ if $y=x$; and equal to 0 otherwise). The restriction of this matrix to the 
rows $x\in\{x_{1,n},\ldots,x_{n,n}\}$ will have nonzero elements in columns
$y\notin \{x_{1,n},\ldots,x_{n,n}\}$ with $y\sim x_{i,n}$ for some $i$, and hence a restriction of $Lf$ to the design points may not
correspond to simply taking the appropriate $(n\times n)$-submatrix of $L$. This is typically solved
by imposing boundary conditions, much as when considering a continuous partial differential operator.

In the example of $\X=\ZZ$ with the design points
$x_{1,n}, \ldots, x_{n,n}$ identified with the points $1,\ldots, n$ and the neighbourhood system:
$i\sim j$ if and only if $|i-j|=1$, the discrete Laplacian is
$$L(f)(i)=\sum_{j: |j-i|=1} \bigl[f(j)-f(i)\bigr]=f(i+1)+f(i-1)-2 f(i).$$
The restrictrion of $L(f)$ to the design points $1,\ldots, n$ also involves the points $0$ and $n+1$,
and there are various ways of imposing boundary conditions. The natural choice $f(0)=f(n+1)=0$
is known as the \emph{Dirichlet boundary}, while the other natural choice $f(0)=f(1)$ and $f(n+1)=f(n)$ is
the \emph{Neumann boundary}. The eigenvectors and eigenvalues corresponding to these boundary
conditions are known explicitly, and so  they are for the mixed Dirichlet-Neumann conditions:
$f(0)=0$ and $f(n+1)=f(n)$.  In fact, in the latter case the eigenvectors are exactly equal to
$e_{j,n}$ as given in (\ref{EqEigenvectors}) and the eigenvalues are $-1/((n+1/2)\l_{j,n})$ for $\l_{j,n}$
as given in (\ref{EqEigenvaluesBM}). 
%http://en.wikipedia.org/wiki/Eigenvalues_and_eigenvectors_of_the_second_derivative
This close connection to Brownian motion is not obvious,
but also not entirely surprising as minus the inverse Laplacian (the twofold primitive) is the covariance operator of 
Brownian motion (restricted to the orthocomplement of the constant functions) and standard
Brownian motion is tied at zero. The connection invites to interpret the eigenvectors 
(\ref{EqEigenvectors}) as modelling smoothness in a discrete sense, an interpretation
that also makes sense if the design points $x_{i,n}$ are linearly ordered and roughly equally
spaced, but not exactly equal to $i/(n+1/2)$ as in Example~\ref{ExampleBrownianMotion}. 
For the special grid of the latter example the norm in (\ref{EqDefFnalpha}) corresponds exactly to the
size measured by the Laplacian, in that 
$$\frac1n\,\|(n^2L)^{\a}\vec f_n\|^2=n^{2\a-1}\sumin
\frac{f_{i,n}^2}{\bigl((n+1/2)\l_{i,n}\bigr)^\a}\asymp \|f\|_{n,\a}^2.$$
(The norm on the left side is the Euclidean norm of $\RR^n$ and the leading factor $1/n$ stabilizes
the sum involved in this norm; the factor $n^2$ preceding $L$ corresponds to $1/h^2$, for $h\sim
1/n$ the mesh width of the grid.) Although the eigenvalues (\ref{EqEigenvaluesBM}) come
naturally with the discrete Laplacian, when defining the prior they might be replaced
by eigenvalues (\ref{EqEigenvalues}) for a general $m$. This would correspond to describing
a-priori smoothness by a power of the Laplacian. Indeed, as noted following (\ref{EqEigenvalues}), for these
eigenvalues we have $\E\|W\|_{n,\a}^2<\infty$ for $\a<(m-1)/2$. In view of the preceding display,
this is equivalent to finiteness of $\frac1n\,\E\|(n^2L)^{\a}\vec W_n\|^2$. So the prior with
covariance matrix (\ref{EqSpectralDecompositionU}), for eigenvalues (\ref{EqEigenvalues}) and
eigenvectors (\ref{EqEigenvectors}), corresponds to modelling $f$ by a Gaussian process $W$
with finite discrete Laplacian $(n^2L^\a) W$ for $\a<(m-1)/2$. 
\end{example}

\begin{example}
[Integrated Brownian motion]
Once integrated Brownian motion $W^1_t=\int_0^t B_s\,ds$, for $B$ standard Browian motion, possesses
covariance function $\cov(W_s^1,W_t^1)=s^2(3t-s)/6$ for $s\le t$. The eigenfunctions are given by
$$e_j(t)\propto (\sin\theta_j+\sinh\theta_j)\bigl(\cos(t\theta_j) - \cosh(t\theta_j)\bigr) 
-(\cos\theta_j+\cosh\theta_j)\bigl(\sin(t\theta_j)-\sinh(t\theta_j)\bigr),$$
where the $\theta_j$ are the positive roots of the equation $\cos (\theta)\cosh(\th)=-1$, for $j\in\{1,2,\ldots\}$.
ee \cite{Freedman}, Theorem~7. The corresponding eigenvalues are $\l_j=\th_j^{-4}$ and are of the order $((2j-1)\pi/2)^{-4}$.

Thus this example appears to satisfy (\ref{EqEigenvalues}) with $m=4$.
However, exact expressions for the discrete eigenvectors and eigenvalues appear not known.
\end{example}

\begin{example}
[Two-dimensional Brownian motion and variants]
Functions $f: [0,1]^2\to \RR$ on the unit square may be modelled a-priori by a scaling of two-dimensional
Brownian motion  $W^1=(W^1_{s,t}: (s,t)\in [0,1]^2)$, which is the tensor product $W^1_{s,t}=B_{1,s}B_{2,t}$ of two 
independent standard univariate Brownian motions $B_1$ and $B_2$. The covariance function $\E W_{s,t}^1W_{s',t'}^1$ is the 
tensor product $K(s,s')K(t,t')$ of the covariance functions $K(s,s')=s\wedge s'$ of the univariate
Brownian motions. For a rectangular grid consisting of points $(x_{i,n}, x_{j,n})$ constructed from a given univariate grid
$0\le x_{1,n}<\cdots< x_{n,n}\le 1$, the covariance matrix of the $n^2$-dimensional vector
$(W_{x_{i,n},x_{j,n}})$, for $(i,j)\in \{1,\ldots,n\}^2$, with its coordinates ordered appropriately, is the Kronecker product of two copies of the covariance
matrix  of the $n$-dimensional vector $(B_{x_{i,n}})$. The eigenvectors are 
the tensor products $e_{i,n}\otimes e_{j,n}$ of the univariate eigenvectors $e_{i,n}$, with corresponding
eigenvalues the products $\l_{i,j,n}=\l_{i,n}\l_{j,n}$ of the univariate eigenvalues $\l_{i,n}$. 

Even though in this case the eigenfunctions and eigenvalues are more naturally viewed as a
two-dimensional array than a sequence, they may of course be ordered in a sequence.
Then this example fits the general setup, except that $n$ has been changed into $n^2$.

In particular, for the grid in Example~\ref{ExampleBrownianMotion} the eigenvectors are 
the discretisations of the tensor products of the sine-basis given in (\ref{EqEigenvectors})
and the eigenvalues satisfy
\begin{equation}
\label{EqEigenvaluesDouble}
\l_{i,j,n}\asymp \frac{n^2}{i^mj^m},\qquad (i,j)\in\{1,\ldots, n\}^2
\end{equation}
for $m=2$. Theorem~\ref{thm:honest}, which assumes  (\ref{EqEigenvalues}), does not
apply to this example. However, the assumptions of the general results below are satisfied,
also for a general value of $m\ge 1$, and hence the message of the theorem goes through. The set of polished tail functions
can be defined in the same manner by (\ref{eq:EqPolishedTail}), after ordering the array
of coefficients $f_{i,j,n}$ in a sequence by order of decreasing eigenvalues $\l_{i,j,n}$ (that is, increasing values of $ij$).

The square smoothness norm $\|\cdot\|_{n,\a}$ as in (\ref{EqDefFnalpha}) now becomes 
$n^{-2}\sumin\sumjn(ij)^\a f_{i,j,n}^2$. While the eigenbasis is essentially the natural two-dimensional Fourier basis,
the restriction imposed by this norm is a bit unusual, in its focus on the cross product $ij$. As the smoothness norm
describes the prior process, this may be unsatisfactory. More natural ``Sobolev norms'' 
$n^{-2}\sumin\sumjn(i^2+j^2)^{\a} f_{i,j,n}^2$ correspond to the eigenvalues
\begin{equation}
\label{EqEigenvaluesDoubleSobolev}
\l_{i,j,n}\asymp \frac{n^2}{(i^2+j^2)^{m}},\qquad (i,j)\in\{1,\ldots, n\}^2.
\end{equation}
The Gaussian process $W^1$ corresponding to these eigenvalues has $\E \|W^1\|_{n^2,\a}^2<\infty$ for every $\a<m-1$,
and hence may be considered ``Sobolev smooth almost of order $m-1$''.

For these eigenvalues the discrete polished tail condition (\ref{eq:generalizedpolishedtail}) can be written in the form
$$\mathop{\sumin\sumjn}_{i^2+j^2\ge m}f_{i,j,n}^2\le L\mathop{\sumin\sumjn}_{m\le i^2+j^2\le \rho m}f_{i,j,n}^2,$$ for sufficiently large $m$. The theorems below show that the credible sets corresponding to this prior cover
functions that satisfy this condition.
\end{example}

\subsection{Organization of the paper}
The paper is structured as follows. In Section~\ref{SectionDs} we analyse the estimators $\hat{c}_n$
of $c$ and next prove our main results about
the coverage of the empirical Bayes credible sets (Section~\ref{SectionCov}). We follow up with
results about contraction rates of oracle type and over various concrete models (Section~\ref{sec:contraction}). 
In Section~\ref{SectionCovHB} we study the
hierarchical Bayes method, starting with the concentration of the posterior distribution of the scaling
parameter $c$ and next using this to determine coverage and contraction. Section~\ref{SectionPolishedTail} concerns the interpretation
of the polished tail condition, which is related to a similar condition on the Fourier coefficients of $f$. It is shown to be satisfied with probability one under the prior. This section also discusses 
various alternative smoothness assumptions on the function $f$. Section~\ref{SectionDiscussion}
is a closing discussion, which addresses conditions, interpretations, and generalizations of our results.
Finally Sections~\ref{SectionProofs} and~\ref{SectionTechnicalResults} gather technical 
proofs and technical lemmas.

\subsection{Notation}
The notation $a_n\asymp b_n$ means that $a_n/b_n$ is bounded away from 0 and infinity, as $n\ra\infty$, and 
$a_n\sim b_n$ means that $a_n/b_n$ tends to 1.
If $a_n$ and $b_n$ are functions, then we say that $a_n\asymp b_n$ or $a_n\sim b_n$ uniformly over a domain 
if the constants away from 0 and infinity can be chosen the same
for every value in the domain, or the convergence to 1 is uniform.

The notation $a\lesssim b$ means $a\le Cb$ for a universal constant $C$.

For a function $g: \X\to\RR$, the vector $\bigl(g(x_{1,n}),\ldots, g(x_{n,n})\bigr)$ is denoted by $\vec g_n$.
The same notational device is used for a vector $\vec\e_n$ composed of  variables $\e_{1,n},\ldots,\e_{n,n}$.

Unless stated otherwise the set $I_n$ is the interval $I_n=[\log n/n,n^{m-1}]$.

\section{Empirical Bayes}
\label{SectionDs}
By substituting the model equation $\vec Y_n=\vec f_n+\vec \e_n$, we can decompose the quadratic forms in
the empirical Bayes criteria (\ref{EqLikEB}) and (\ref{EqRIskEB}) as
\begin{equation}\label{eq:yS2y}
\vec Y_n^T \Sigma_{n,c}^{-k}\vec Y_n= \vec{f}_n^T \Sigma_{n,c}^{-k} \vec{f}_n + \vec\e_n^T \Sigma_{n,c}^{-k} \vec\e_n 
+ 2 \vec{f}_n^T \Sigma_{n,c}^{-k}\vec \e_n, \qquad k\in\{1,2\}.
\end{equation}
We next express both $\vec f_n$ and $\vec \e_n$ relative to the orthonormal eigenbasis $e_{1,n},\ldots,e_{n,n}$ of $U_n$.
The coefficients of $\vec f_n$ are by their definition the numbers $ f_{j,n}$, while the coefficients of
$\vec \e_n$ are i.i.d.\ standard normal variables $Z_{j,n}$. The matrix $\Sigma_{n,c}=I+c U_n$ and its inverses
$\Sigma_{n,c}^{-1}$ and $\Sigma_{n,c}^{-2}$ have the same eigenbasis as $U_n$, 
with eigenvalues $(1+c\lambda_{j,n})$, $(1+c\lambda_{j,n})^{-1}$ and $(1+c\lambda_{j,n})^{-2}$, respectively, for
$\lambda_{j,n}$ the eigenvalues of $U_n$. It follows that the two types 
of empirical Bayes estimators $\hat c_n$ minimize criteria $L_n^L$ and $L_n^R$ of the form
\begin{align}
\label{EqDefLn}
L_n(c,f)&:=D_{1,n}(c,f)+D_{2,n}(c)+R_{1,n}(c,f)+R_{2,n}(c)\\
&= D_n(c,f)+R_n(c,f).\nonumber
\end{align}
For the risk-based empirical Bayes estimator (\ref{EqRIskEB}) the functions and processes $D_{1,n}, D_{2,n}, R_{1,n}$ and $R_{2,n}$ 
on the right side are defined by
\begin{equation}
\label{EqDandRRisk}
\begin{aligned}
D_{1,n}^R(c,f)&=\vec{f}_n^T \Sigma_{n,c}^{-2} \vec{f}_n=\sum_{j=1}^n \frac{f_{j,n}^2}{(1+c\lambda_{j,n})^2},\\
D_{2,n}^R(c) &= \tr\bigl((I-\Sigma_{n,c}^{-1})^2\bigr)=\sum_{j=1}^n\frac{(c\lambda_{j,n})^2}{(1+c\lambda_{j,n})^2},\\
R_{1,n}^R(c,f)& = 2\vec f_n^T\Sigma_{n,c}^{-2}\vec \e_n=2\sum_{j=1}^n \frac{Z_{j,n}f_{j,n}}{(1+c\lambda_{j,n})^2}, \\
R_{2,n}^R(c) &= \vec\e_n^T\Sigma_{n,c}^{-2}\vec\e_n-\tr(\Sigma_{n,c}^{-2}) - \sum_{j=1}^n (Z_{j,n}^2-1)
=\sum_{j=1}^n (Z_{j,n}^2-1)\Bigl[\frac{1}{(1+c\lambda_{j,n})^2}-1\Bigr],
\end{aligned}
\end{equation}
whereas for the likelihood-based empirical Bayes estimator (\ref{EqLikEB}) these functions and processes are given by
\begin{equation}
\label{EqDandRLik}
\begin{aligned}
D_{1,n}^L(c,f)&=\vec{f}_n^T \Sigma_{n,c}^{-1} \vec{f}_n=\sum_{j=1}^n \frac{f_{j,n}^2}{1+c\lambda_{j,n}},\\
D_{2,n}^L(c) &= \log\det\Sigma_{n,c}-\tr\bigl(I-\Sigma_{n,c}^{-1}\bigr)=\sumjn\Bigl[\log(1+c\lambda_{j,n})-\frac{c\l_{j,n}}{1+c\lambda_{j,n}}\Bigr],\\
R_{1,n}^L(c,f)& = 2\vec f_n^T\Sigma_{n,c}^{-1}\vec \e_n=2\sum_{j=1}^n \frac{Z_{j,n}f_{j,n}}{1+c\lambda_{j,n}}, \\
R_{2,n}^L(c) &= \vec\e_n^T\Sigma_{n,c}^{-1}\vec\e_n-\tr(\Sigma_{n,c}^{-1}) - \sum_{j=1}^n (Z_{j,n}^2-1)
=-\sum_{j=1}^n \frac{(Z_{j,n}^2-1)c\lambda_{j,n}}{1+c\lambda_{j,n}}.
\end{aligned}
\end{equation}
In general discussions we shall leave off the superscripts $R$ and $L$, for ``Risk'' and ``Likelihood'', and denote both
the risk- and likelihood-based functions by $D_{1,n}, D_{2,n}, R_{1,n}, R_{2,n}$. In both cases we have
shifted the criteria by the factor $\sumjn (Z_{j,n}^2-1)$, which does not depend on $c$, 
in order that the remainder term $R_{2,n}$ be smaller.

The functions $D_{1,n}$ and $D_{2,n}$ are deterministic, whereas $R_{1,n}$ and $R_{2,n}$ are random
processes. The processes $D_{1,n}$ and $R_{1,n}$ depend on $f$, whereas the other processes are free of the parameter. Even
though the functions and processes differ in the risk- and likelihood-based cases, for instance by
the power of $1+c\lambda_{n,j}$ in the denominators, the two estimators $\hat c_n$ can be analysed
by similar methods.  In Lemma~\ref{prop:D2nieuw} it will be seen that under (\ref{EqEigenvalues}) 
the two functions $D_{2,n}$, even though quite different in form, are
asymptotically equivalent.  The following proposition shows that in both cases the
stochastic process $R_n$ is negligible relative to the deterministic process $D_n$.

\begin{prop}\label{prop:R}
If (\ref{EqEigenvalues}), (\ref{EqEigenvaluesDouble}) or (\ref{EqEigenvaluesDoubleSobolev}) holds, 
then for $R_{1,n}$ and $R_{2,n}$ as given in (\ref{EqDandRRisk}) or (\ref{EqDandRLik}) 
and the corresponding $D_n=D_{1,n}+D_{2,n}$ in the same display it holds that
\begin{equation}
\label{EqRemainderR1R2}
\sup_{c\in \interval } \frac{|R_{1,n}(c,f)|+|R_{2,n}(c)|}{D_n(c,f)}\stackrel{P_f}{\ra} 0.
\end{equation}
\end{prop}

\noindent The proof of the proposition can be found in Section~\ref{SectionProofs}.
In case of the eigenvalues (\ref{EqEigenvaluesDouble}) or (\ref{EqEigenvaluesDoubleSobolev}), it should be
understood that $n$ is replaced by $n^2$ in the assertion (and the single sums in 
(\ref{EqDandRRisk}) or (\ref{EqDandRLik}) by double sums).

We may view the stochastic process $R_n=R_{1,n}+R_{2,n}$ in (\ref{EqDandRRisk}) or (\ref{EqDandRLik})  
as an ``estimation error'' when estimating an  ``ideal''  criterion $D_n=D_{1,n}+D_{2,n}$. The preceding
proposition essentially says that this error can be ignored. 
As a consequence the minimizer $\hat c_n$ of $L_n=D_n+R_n$ will behave similarly to the
(deterministic) minimizer of $D_n$. The latter functions consists of a part $D_{1,n}(\cdot,f)$ that 
is decreasing in $c$, from $D_{1,n}(0,f)=\sumjn f_{j,n}^2$ to $D_{1,n}(\infty, f)=0$, 
and a part $D_{2,n}$ that is free of $f$ and is strictly increasing in $c$, from $D_{2,n}(0)=0$ to $D_{2,n}(\infty)\ge n$.
Minimizing the sum $D_n$ of these functions can be viewed as an attempt to balance these two terms.

In the case of the risk-based empirical Bayes method $D_{1,n}(c,f)$ is exactly the square bias of the 
posterior mean at the true regression function $f$, given a fixed scale $c$, and $D_{2,n}(c)$
is its variance, which is independent of $f$ (see (\ref{EqRiskPostMean})). The square bias is decreasing in the scale $c$, while
the variance is increasing, and hence the empirical Bayes estimator $\hat c_n$ tries to balance the square bias and variance
by minimizing an estimate of their sum.
The likelihood-based empirical Bayes estimator is not as strongly tied to the risk,
but we shall see that it performs in a similar manner. Here the essence will be that its bias term
$D_{1,n}$ is bigger than the bias term of the risk-based method, while its variance term has the same order of magnitude.

For minimizing the risk the empirical Bayes methods always  do the right thing. However, the coverage of the
credible sets depends not on the sum of square bias and variance, but on their relationship, or rather
the relationship between square bias and the \emph{posterior variance}
\begin{equation}
\label{EqPosteriorVariance}
s_n^2(c)=\E \bigl(\|\vec f_n-\hat f_{n,c}\|^2\given \vec Y_n,c\bigr)
=\tr(I-\Sigma_{n,c}^{-1})=\sumjn \frac{c\l_{j,n}}{1+c\lambda_{j,n}}.
\end{equation}
If for a particular $f$ the square bias exceeds the posterior variance, then the empirical Bayes method 
will put a too narrow credible set too far from the truth, which it will not cover in that case. The posterior variance,
although not equal to the variance terms $D_{2,n}$, 
has the same order of magnitude as these quantities (see Lemma~\ref{prop:D2nieuw}). Thus a lack of coverage is caused by
too small a value of $\hat c_n$, giving too small a prior variance and posterior variance, i.e.\ by ``oversmoothing'' the truth.

Notwithstanding the nice properties of the functions $D_{1,n}$ and $D_{2,n}$ for a given $n$,
such oversmoothing may occur for $f$ for which the ``bias'' function $c\mapsto D_{1,n}(c,f)$ changes
haphazardly with $n$. (We describe this here in an asymptotic framework, with $n\ra\infty$, but
a problem will arise for every given $n$, albeit possibly for different $f$.) The point is that at different sample sizes,
different aspects of $f$ determine the behaviour of the empirical Bayes estimators $\hat c_n$.
The assumption that $f$ satisfies the polished tail condition prevents such haphazard behaviour for both empirical Bayes
methods. When considering a given method, good behaviour can also be more precisely 
characterised through the corresponding function $D_{1,n}$, as follows.

\begin{dfn}[Good bias condition]
\label{def:goodbias}
We say that the function $f$, or the corresponding array $(f_{j,n})$, satisfies the \emph{good bias condition} relative to
$D_{1,n}$ if there exists a constant $a>0$ such that, for $c\in\interval$,
\begin{equation}
\label{eq:EqGoodBias}
D_{1,n}(Kc,f)\le K^{-a} D_{1,n}(c,f),\qquad \text{ for all }K>1.
\end{equation}
\end{dfn}

As a pendant to this condition we call $D_{2,n}$ \emph{good variance functions} if there exist constants $b, B,B >0$,
independent of $n$, such that for $c\in\interval$ we have
\begin{equation}
\label{EqGoodVariance}
B k^{b} D_{2,n}(c) \le D_{2,n}(kc)\le B' k^b D_{2,n}(c) \qquad \text{ for all }\quad k<1.
\end{equation}
Since the functions $D_{2,n}$ do not depend on $f$, the good variance condition merely refers to the prior process. 
Priors satisfying (\ref{EqEigenvalues}) give $D_{2,n}(c)\asymp (cn)^{1/m}$
(see Lemma~\ref{prop:D2nieuw}) and hence yield good variance functions with $b=1/m$.

The essence of these ``good conditions'' is captured in the purely analytical Lemma~\ref{LemmaCrossing}
in Section~\ref{SectionTechnicalResults}, which is the basis of the proof of the second assertion of the following theorem.

\begin{thm}
\label{thm:schatter}
Suppose the remainder terms $R_{1,n}$ and $R_{2,n}$ satisfy (\ref{EqRemainderR1R2}). Then for any $f$ and $\e>0$ the empirical Bayes estimators $\hat c_n$ given in (\ref{EqLikEB}) and (\ref{EqRIskEB}), with the corresponding function $D_n=D_{1,n}+D_{2,n}$ 
as given in (\ref{EqDandRRisk}) and (\ref{EqDandRLik}), satisfy
$$P_f\Bigl(D_n(\hat{c}_n,f)\le (1+\e)\inf_{c\in \interval} D_n(c,f)\Bigr)\ra 1.$$
Furthermore, if $f$ satisfies the good bias condition with constant $a$, $D_{2,n}$ are good variance functions with constants $b,B,B'$ and $\sumjn f_{j,n}^2\le \sup_{c\in I_n}D_{2,n}(c)$, then also
$$P_f\Bigl(D_{1,n}(\hat{c}_n,f)\le B^{-1}(2+2\e)^{1+b/a} D_{2,n}(\hat c_n)\Bigr)\ra 1.$$
\end{thm}

\begin{proof}
Let $c_n\in\interval$ be a minimizer of $D_n$ and set $\Lambda_n=\{c\in \interval: D_n(c,f)\le (1+\e) D_n(c_n,f)\}$.
For the first assertion it suffices to show that $P_f(\hat c_n\in \Lambda_n)\ra 1$. By the definition of $\hat c_n$,
this is the case if $\inf_{c\notin\Lambda_n}L_n(c,f)$ is with probability tending to one strictly bigger than $L_n(c_n,f)$.
Since $L_n=D_n+R_n$, relation  (\ref{EqRemainderR1R2}) gives
\begin{align*}
\inf_{c\notin\Lambda_n}L_n(c,f)
&= \inf_{c\notin\Lambda_n} \Bigl[D_n(c,f)\Bigl(1+\frac{R_n(c,f)}{D_n(c,f)}\Bigr)\Bigr]\\
&\ge \inf_{c\notin\Lambda_n} D_n(c,f)\Bigl(1-\sup_{c\notin\Lambda_n}\Bigl|\frac{R_n(c,f)}{D_n(c,f)}\Bigr|\Bigr)
\ge \Bigl[\inf_{c\notin\Lambda_n} D_n(c,f)\Bigr] \bigl(1-o_P(1)\bigr)
\end{align*}
By the definition of $\Lambda_n$ the infimum on the right side is at least $(1+\e)D_n(c_n,f)$. 
Moreover, again by Proposition~\ref{prop:R} we have that $L_n(c_n,f)\le D_n(c_n,f)\bigl(1+o_P(1)\bigr)$.
The desired result follows, as $D_n(c_n,f)$ is strictly positive.

For the proof of the second assertion we define $\tilde c_n$ as the unique
point of intersection of the graphs of the functions $D_{1,n}$ and $D_{2,n}$, 
i.e.\ the unique solution of the equation $D_{1,n}(c,f)=D_{2,n}(c)$.
If $\tilde c_n\in \interval$, then by the first assertion $D_n(\hat c_n,f)\le (1+\e) D_n(\tilde c_n,f)$, whence the assertion follows
from Lemma~\ref{LemmaCrossing}(i). If $\tilde c_n$ falls to the left of $\interval$, then $D_{1,n}(c,f)\le D_{2,n}(c)$ throughout
$\interval$ by the monotonicity of the two functions and the assertion is trivially true. 
The assumption that $D_{1,n}(0,f)=\sumjn f_{j,n}^2$ is below the maximum
value of $D_{2,n}$ prevents that $\tilde c_n$ falls to the right of $\interval$.
\end{proof}

The good-bias condition on $f$ is dependent on the prior and the method through the function $D_{1,n}$, which
can be $D_{1,n}^L$ or $D_{1,n}^R$. For both methods the condition is implied by the discrete polished tail condition.

\begin{lem}\label{lem:staartweg}
Any $f$ that satisfies the discrete polished tail condition also satisfies the good bias
condition, for both the risk-based and likelihood-based bias functions $D_{1,n}(\cdot, f)$.
\end{lem}

\begin{proof}
If $f$ satisfies the discrete polished tail condition, then
\[
  \sum_{j: c\l_{j,n}\le 1} \frac{f_{j,n}^2}{(1+c \lambda_{j,n})^2} 
\leq \sum_{j: c\l_{j,n}\le 1} f_{j,n}^2 
\le L \sum_{j: \rho\le c\l_{j,n}\le 1} f_{j,n}^2 
\le 4L  \sum_{j: \rho\le c\l_{j,n}\le 1} \frac{f_{j,n}^2}{(1+c \lambda_{j,n})^2},
 \]
since $1+c\l_{j,n}\le 2$ for $j$ in the range of the sum. The left side is part of the sum that defines
the function $D_{1,n}^R$.  Splitting this sum in the parts with $c\l_{j,n}\le 1$ and with $c\l_{j,n}>1$ and noting
that $\rho\le 1$, we see
\[
  D_{1,n}^R(c,f)\le (1+4L) \sum_{j: \rho\le c\l_{j,n}}\frac{f_{j,n}^2}{(1+c \lambda_{j,n})^2}  
\le \frac{(1+4L) (1+\rho)}{\rho}  \sum_{j: \rho\le c\l_{j,n}} \frac{f_{j,n}^2 c\lambda_{j,n}}{(1+c\lambda_{j,n})^3},
\]
since $c\lambda_{n,j}/(1+c\lambda_{n,j})\ge \rho/(1+\rho)$ for $j$ in the range of the sum.
The sum on the right side becomes even bigger if we let the sum range from $1$ to $n$ and
is then equal to $-\frac12 c\, (D^R_{1,n})'(c)$. It follows that there exists $a>0$ such that
 \[
  \frac{(D^R_{1,n})'(c,f)}{D_{1,n}^R(c,f)} \le-\frac{a}{c}.
 \]
On integrating this from $c$ to $Kc$ we find that $\log D_{1,n}^R(Kc,f)-\log D_{1,n}^R(c,f)$ is bounded
above by $-a\log K$, and the good bias condition (\ref{eq:EqGoodBias}) follows.

The proof for the likelihood-based function $D_{1,n}^L$ differs only in that the power of the terms $(1+c\l_{j,n})^2$ in the denominator
must be decreased from 2 to 1.
\end{proof}

The following lemma gives the behaviour of the three variance functions if the
eigenvalues satisfy (\ref{EqEigenvalues}), (\ref{EqEigenvaluesDouble}) or (\ref{EqEigenvaluesDoubleSobolev}). 
The lemma implies that these three functions are good variance functions in the sense of (\ref{EqGoodVariance}).

\begin{lem}
\label{lem:trace}
\label{prop:D2nieuw}
The functions $D_{2,n}^R$ given in (\ref{EqDandRRisk}),  $D_{2,n}^L$ given in (\ref{EqDandRLik}) and $s_n$ given in 
(\ref{EqPosteriorVariance}) are strictly increasing on $[0,\infty)$. Furthermore, if
(\ref{EqEigenvalues}) holds, then
$$D_{2,n}^R(c)\asymp D_{2,n}^L(c)\asymp s_n^2(c)\asymp(cn)^{1/m},$$
uniformly in $c$ in $\interval$ as $n\ra\infty$. The same is true (with $n^2$ instead of $n$) under (\ref{EqEigenvaluesDoubleSobolev}). 
Moreover, if (\ref{EqEigenvaluesDouble}) holds, then
$$D_{2,n^2}^R(c)\asymp D_{2,n^2}^L(c)\asymp s_{n^2}^2(c)\asymp
\begin{cases}
(cn^2)^{1/m}\bigl(1+\log (cn^2)\bigr)&\text{ if } cn^2\le n^m,\\
(cn^2)^{1/m}\bigl(1+\log (n^{2m}/(cn^2))\bigr)&\text{ if } cn^2\ge n^m,
\end{cases}
$$
uniformly in $c$ in $I_{n^2}$.%[\log n/n^2, n^{2m-2}]$.
\end{lem}

\begin{proof}
The monotonicity of $D_{2,n}^R$ and $s_n$ is clear. Under (\ref{EqEigenvalues}) the function $D_{2,n}^R$ satisfies
 \[
 (cn\underline\d)^2 \sum_{j=1}^n\frac{1}{(j^{m}+cn\overline\d)^2} 
\leq D_{2,n}^R(c) \leq (cn\overline\d)^2 \sum_{j=1}^n\frac{1}{(j^{m}+cn\overline\d)^2},
 \]
where in the second inequality we use that $x\mapsto {x}/({1+x})$ is increasing. By Lemma~\ref{lem:som} in the appendix the sums are of the order 
$(\overline\delta c n)^{-2+1/m}$ for $c\in\interval$. The function $s_n$ can be treated analogously.

The derivative of $D_{2,n}^L$ is given by
$$(D_{2,n}^L)'(c) = \sumjn \left( \frac{\lambda_{j,n}}{1+c\lambda_{j,n}} - \frac{\lambda_{j,n}}{(1+c\lambda_{j,n})^2}\right)
=\sumjn \frac{c\lambda_{j,n}^2}{(1+c\lambda_{j,n})^2}.$$
The monotonicity of $D_{2,n}^L$ is a consequence of the positivity of this function.
The value of $D_{2,n}^L$ at $c$ is the integral of this derivative over the interval $[0,c]$. If 
(\ref{EqEigenvalues}) holds, then
$$\underline\d^2\int_0^c \sumjn \frac{sn^2}{(j^{m}+\overline\d sn)^2}\,ds
\le D_{2,n}^L(c)\le\overline\d^2\int_0^c \sumjn \frac{sn^2}{(j^{m}+\underline\d sn)^2}\,ds.$$
By Lemma~\ref{lem:som} the integrands are asymptotic to  a multiple of $(sn^2)(\delta s n)^{-2+1/m}
= n^{1/m}s^{-1+1/m}$ uniformly in $s\in [l_n/n, n^{m-1}]$, for any $l_n\ra\infty$
and $\delta=\underline\d$ and $\delta=\overline\d$ respectively. The integral
of the latter function over $[0,c]$ is equal to a multiple of $(c n)^{1/m}$, while its
integral over  $[0,l_n/n]$ is of the order $l_n^{1/m}$. The integral of $(D_{1,n}^L)'$ over
$[0,l_n/n]$ is bounded above by a multiple of $\int_0^{l_n/n}sn^2\sumjn j^{-2m}\,ds\asymp l_n^2$.
Hence both remainders are of lower order than $(cn)^{1/m}$ for $c\in \interval$ if 
$l_n$ is chosen equal to, for instance, $\log\log n$.

The proof under (\ref{EqEigenvaluesDoubleSobolev}) is the same, except that we use Lemma~\ref{lem:dubsom2} instead
of Lemma~\ref{lem:som}. The final assertion also follows along the same lines, but now employing Lemma~\ref{lem:dubsom}.
The details are deferred to Section~\ref{SectionDetailsDoubleSum}.
\end{proof}

\subsection{Coverage of the empirical Bayes credible sets}
\label{SectionCov}
The function $f$ is contained in the empirical Bayes credible sets (\ref{EqCredibleSetEB}) 
if $\|\vec f_n-\hat f_{n,\hat c_n}\|\le M r_n(\hat c_n,\eta)$. 
In view of (\ref{EqPosterior}) and (\ref{eq:RegProblem}), the square of the left side can be decomposed for any $c$ as
\begin{align}
\|\hat{f}_{n,c}-\vec{f}_n\|^2 
&=\vec{f}_n^T\Sigma_{n,c}^{-2}\vec{f}_n- 2 \vec f_n^T\Sigma_{n,c}^{-1}(I-\Sigma_{n,c}^{-1})\vec{\e}_n
+ \vec\e_n^T(I-\Sigma_{n,c}^{-1})^2 \vec \e_n \nonumber\\
&=D_{1,n}^{R}(c,f)+D_{2,n}^{R}(c)+R_{3,n}(c,f)+R_{4,n}(c),\label{EqDecompositionSquareNorm}
\end{align}
where the first two processes on the right are defined in (\ref{EqDandRRisk}) and (\ref{EqDandRLik}) and
\begin{equation}
\label{EqDefR4}
\begin{aligned}
R_{3,n}(c,f)&=- 2 \vec f_n^T\Sigma_{n,c}^{-1}(I-\Sigma_{n,c}^{-1})\vec{\e}_n = -2\sum_{j=1}^n \frac{(c\lambda_{j,n})Z_{j,n}f_{j,n}}{(1+c\lambda_{j,n})^2},\\
R_{4,n}(c)&=\vec\e_n^T(I-\Sigma_{n,c}^{-1})^2 \vec \e_n -\tr\bigl((I-\Sigma_{n,c}^{-1})^2 \bigr)
=\sumjn \frac{(c\lambda_{n,j})^2(Z_{j,n}^2-1)}{(1+c\lambda_{j,n})^2}.
\end{aligned}
\end{equation}
The following proposition shows that the remainder $R_{3,n}+R_{4,n}$ is negligible relative to the deterministic process $D_n$, 
for both the likelihood-based and risk-based functions.

\begin{prop}\label{prop:R4}
If (\ref{EqEigenvalues}), (\ref{EqEigenvaluesDouble}) or (\ref{EqEigenvaluesDoubleSobolev}) holds, then  
for $R_{3,n}$ and $R_{4,n}$ given in (\ref{EqDefR4}) and  
$D_n=D_{1,n}+D_{2,n}$ given in (\ref{EqDandRRisk})  or (\ref{EqDandRLik}) we have
\begin{equation}
\label{EqRemainderR4}
\sup_{c\in \interval }\frac{|R_{3,n}(c,f)|+|R_{4,n}(c)|}{D_n(c,f)} \stackrel{P_f}{\ra} 0.
\end{equation}
\end{prop}

\noindent The proof of the proposition can be found in Section~\ref{SectionProofs}.

The radius $r_n(c,\eta)$ of the Bayesian credible set is the $\eta$-quantile of the posterior distribution of
$\|\vec f_n -\hat{f}_{n,c}\|$ given $c$. As the distribution of $\vec{f}_n -\hat{f}_{n,c}$ does not depend on $Y$,
the radius $r_n(c,\eta)$ is deterministic. Since the posterior distribution of 
$\vec f_n -\hat{f}_{n,c}$ is multivariate normal with mean zero and
covariance matrix $I-\Sigma_{n,c}^{-1}$ (see (\ref{EqPosterior})), the square norm is equal in distribution to 
the variable
\begin{equation}
\label{DefN(c)}
N_n(c)=\sumjn \frac{c\lambda_{j,n}Z_{j,n}^2}{1+c\lambda_{j,n}},
\end{equation} 
where the $Z_{j,n}$ are independent standard normal random variables. 
The mean of this variable is  by its definition the posterior variance $s_n^2(c)$, given in (\ref{EqPosteriorVariance}).
The following proposition shows that the variables $N_n$ degenerate to their mean as $n\ra\infty$.

\begin{prop}
\label{prop:postunif}
If (\ref{EqEigenvalues}), (\ref{EqEigenvaluesDouble}) or (\ref{EqEigenvaluesDoubleSobolev}) holds, then 
\begin{equation}
\label{EqN}
\sup_{c\in \interval }\left| \frac{N_n(c)}{s_n^2(c)}-1\right| \stackrel{P}{\ra} 0.
\end{equation}
\end{prop}

\noindent The proof of the proposition can be found in Section~\ref{SectionProofs}.

We are ready for our main result on coverage. The result applies to discrete polished tail functions
and under every of the three eigenvalues conditions, but we give a more general statement, which takes the output of the
preceding propositions as its conditions.

\begin{thm}
[Coverage]
\label{TheoremCoverageEB} 
Suppose the following conditions hold:
\begin{enumerate}
 \item the remainders $R_{1,n}$ and $R_{2,n}$ behave as in (\ref{EqRemainderR1R2}) and $R_{3,n}$ and $R_{4,n}$ behave as in (\ref{EqRemainderR4}),
 \item (\ref{EqN}) is satisfied,
 \item $D_{2,n}^R(c)\asymp D_{2,n}^L(c)\asymp s_n^2(c)$ uniformly in $c\in\interval$,
 \item the function $f$ satisfies the good bias condition and $\sumjn f_{j,n}^2\le \sup_{c\in I_n}D_{2,n}(c)$.
\end{enumerate}
Then $P_f(f\in \hat C_{n,\eta,M})\ra 1$, for both the risk-based and likelihood-based
credible sets (\ref{EqCredibleSetEB}) and sufficiently large $M$. In particular, this is true 
If (\ref{EqEigenvalues}), (\ref{EqEigenvaluesDouble}) or (\ref{EqEigenvaluesDoubleSobolev}) and condition 4 above hold. 
\end{thm}

\begin{proof}
Since $N_n(c)/s_n^2(c)\ra 1$ in probability uniformly in $c\in \interval$, the quantities $r_n^2(c,\eta)/s_n^2(c)$, which are the $\eta$-quantiles of the variables $N_n(c)/s_n^2(c)$, tend to 1 as well, uniformly in $c$.
In order to see this, suppose that $\sup_{c\in\interval} |r_n^2(c,\eta)/s_n^2(c)-1| \not \to 0$. Then there exist a subsequence $r_{n_k}^2/s_{n_k}^2$ and points $c_k\in\interval$ such that $|r_{n_k}^2(c_k,\eta)/s_{n_k}^2(c_k)-1|>\epsilon$. We may assume that we either have $r_{n_k}^2(c_k,\eta)/s_{n_k}^2(c_k)>1+\epsilon$ for all $k$ or $r_{n_k}^2(c_k,\eta)/s_{n_k}^2(c_k)<1-\epsilon$ for all $k$. In the latter case, we see that along this subsequence we have
\[
  \P\left(  \frac{N_{n_k}(c_k)}{s_{n_k}^2(c_k)} < \frac{r_{n_k}^2(c_k,\eta)}{s_{n_k}^2(c_k)} \right) \leq \P\left(\sup_{c\in I_{n_k}} \frac{N_{n_k}(c)}{s_{n_k}^2(c)} < 1-\epsilon \right)\to 0
\]
by (\ref{EqN}). The case that $r_{n_k}(c_k)>1+\epsilon$ can be treated similarly, where now this probability tends to one. In either case, this contradicts the definition of $r_n(c,\eta)$. 

It follows that the function $f$ is contained in $\hat C_{n,\eta, M}$ if 
$\|\hat{f}_{n,\hat c_n}-\vec{f}_n\|^2/s_n^2(\hat c_n)\le M^2\bigl(1+o_P(1)\bigr)$.
By the decomposition (\ref{EqDecompositionSquareNorm}) this is equivalent to
$$\frac{D_{1,n}^{R}(\hat c_n,f)+D_{2,n}^{R}(\hat c_n)+R_{3,n}(\hat c_n,f)+R_{4,n}(\hat c_n)}{s_n^2(\hat c_n)}
\le M^2\bigl(1+o_P(1)\bigr).$$
By assumption $s_n^2(\hat c_n)$ has the same asymptotic behaviour
as both $D_{2,n}^R(\hat c_n)$ and $D_{2,n}^L(\hat c_n)$, up to a multiplicative constant. If $f$ satisfies the good bias condition
for the risk-based procedure, then $D_{2,n}^R(\hat c_n)\gtrsim D_{1,n}^{R}(\hat c_n,f)$ with
probability tending to one by Theorem~\ref{thm:schatter}, whence
$D_{n}^R(\hat c_n,f)\asymp D_{2,n}^R(\hat c_n)\asymp s_n^2(\hat c_n)$.
It then follows that the first two terms in the display are bounded above,
while the remainder terms tend to zero by (\ref{EqRemainderR4}). 

By definition we always have $D_{1,n}^R(c,f)\le D_{1,n}^L(c,f)$. 
If $f$ satisfies the good bias condition for the likelihood-based procedure, then 
$D_{1,n}^L(\hat c_n,f)\lesssim D_{2,n}^{L}(\hat c_n)$ with probability tending to one by Theorem~\ref{thm:schatter},
while $D_{2,n}^{L}(\hat c_n) \asymp D_{2,n}^{R}(\hat c_n)$ by assumption.
It follows that again $D_{1,n}^R(\hat c_n,f)\lesssim D_{2,n}^R(\hat c_n)$,
and the proof is analogous to the risk-based case, where for the last two terms we use the fact that $D_{n}^L(\hat c_n,f)\asymp D_{2,n}^L(\hat c_n)\asymp s_n^2(\hat c_n)$.

The final assertion of the theorem follows by Propositions~\ref{prop:R}, \ref{prop:R4} and~\ref{prop:postunif}
and Lemma~\ref{prop:D2nieuw}, which show that all assumptions hold
under (\ref{EqEigenvalues}), (\ref{EqEigenvaluesDouble}) or (\ref{EqEigenvaluesDoubleSobolev}) and the conditions on $f$.
\end{proof}

\subsection{Contraction rates of the empirical Bayes posteriors}
\label{sec:contraction}
We first consider the risk-based setting. If the remainder processes in (\ref{EqDecompositionSquareNorm}) are negligible relative
to $D_n^R=D_{1,n}^R+D_{2,n}^R$ uniformly in $c\in\interval$, which is true under our three eigenvalue conditions by Proposition~\ref{prop:R4}, then 
\begin{equation}
\label{EqBoundBayesDistancePostmean}
\|\hat f_{n,\hat c_n}-\vec f_n\|^2= O_P\bigl( D_n^R(\hat c_n,f)\bigr).
\end{equation}
For the estimator $\hat c_n$ the right side is by the first assertion of Theorem~\ref{thm:schatter} of the order (in probability)
$$\inf_{c\in\interval} D_n^R(c,f)$$
with probability tending to one. Since $D_n^R(c,f)$ is exactly the risk of
the estimator $\hat f_{n,c}$ for a given $c$, these two assertions combined can be viewed as 
an \emph{oracle type} inequality for the risk-based empirical Bayes plug-in posterior mean  $\hat f_{n,\hat c_n}$:
the empirical Bayes estimator manages to choose the best value of $c$ for each possible $f$.
The family of estimators $\hat f_{n,c}$, where $c\in\interval$, turns out be rich enough to give an optimal estimation rate
for the usual regularity classes. Thus the estimator $\hat f_{n,\hat c_n}$ adapts to unknown regularity
in the usual sense. We formalize this in the next theorem, together with the observation
that the posterior variance also adapts correctly. From this we deduce that the full posterior distribution contracts
adaptively.

Write $\Pi_c\bigl(\cdot\given \vec Y_n\bigr)$ for the posterior distribution of $\vec f_n$ given $c$ and let 
$\Pi_{\hat c_n}\bigl(\cdot\given \vec Y_n\bigr)$ be the same object, but with $c$ replaced by $\hat c_n$.

\begin{thm}
[Contraction, risk-based EB]
\label{thm:deconcluderendeconclusie}Suppose the following conditions hold:
\begin{enumerate}
 \item the remainders $R_{1,n}$ and $R_{2,n}$ behave as in (\ref{EqRemainderR1R2}) and $R_{3,n}$ and $R_{4,n}$ behave as in (\ref{EqRemainderR4}),
 \item $D_{2,n}^R(c)\asymp s_n^2(c)$ uniformly in $c\in\interval$.
 \end{enumerate}
 Then  for $\hat c_n$ given by (\ref{EqRIskEB}) and any sequence $M_n\ra\infty$,
\[
 \Pi_{\hat{c}_n}\Bigl(w: \|\vec{w}_n-\vec{f}_n\|^2\geq M_n \inf_{c\in\interval} \E_{f}\|\hat f_{n,c}-\vec f_n\|^2\given \vec Y_n\Bigr) 
\stackrel{\P_{f}}{\ra} 0.
 \]
In particular, this is true if (\ref{EqEigenvalues}), (\ref{EqEigenvaluesDouble}) or (\ref{EqEigenvaluesDoubleSobolev}) holds.
\end{thm}

\begin{proof}
Let $W$ denote a variable that given $\vec Y_n$ and $c$ is distributed according
to the posterior distribution of $f$. Then by Markov's inequality, for any $M$ and $c$,
\begin{align*}
M^2\,\Pi_{c}\bigl(w: \|\vec{w}_n-\vec{f}_n\|^2\geq M^2\given \vec Y_n\bigr) 
&\leq \E \bigl[\|\vec{W}_n-\vec{f}_n\|^2 \given \vec Y_n,c\bigr]\\
&\leq\|\hat{f}_{n,c}-\vec{f}_n\|^2+ \E\bigl[\|\vec{W}_n-\hat{f}_{n,c}\|^2 \given \vec Y_n,c\bigr].
\end{align*}
The second term on the far right is the posterior variance $s_n^2(c)$, which by assumption is bounded by a multiple
of $D_{2,n}^R(c)\le D_n^R(c,f)$ uniformly in $c\in\interval$. 
The first term on the far right evaluated at $c=\hat c_n$ is bounded above by $D_n^R(\hat c_n,f)$ with probability tending to one, in view of 
(\ref{EqDecompositionSquareNorm}) and (\ref{EqRemainderR1R2}) and (\ref{EqRemainderR4}).
It follows that with probability tending to one
$$\Pi_{\hat c_n}\bigl(w :\|\vec{w}_n-\vec{f}_n\|^2\geq M^2\given \vec Y_n\bigr) 
\lesssim \frac 1{M^2}D_n^R(\hat c_n,f)
\lesssim \frac 1{M^2}\inf_{c\in\interval} D_n^R(c,f)$$
by the first assertion of Theorem~\ref{thm:schatter}. Since $D_n^R(c,f)=\E_{f}\|\hat f_{n,c}-\vec f_n\|^2$,
the proof is complete.
\end{proof}

Thus the risk-based empirical Bayes method attains a rate of contraction equal to
the best estimator in the class of estimators $\hat f_{n,c}$, for $c\in \interval$. In standard
models this class contains a rate-minimax estimator.

\begin{example}
[Sobolev norm]
\label{ex:ExampleSobolev}
Denote by $S_n^\a$ the set all functions $f$ for which the discrete Sobolev norm $\|f\|_{n,\a}$, defined
in (\ref{EqDefFnalpha}), is bounded by 1. 
For eigenvalues satisfying (\ref{EqEigenvalues}) and $f\in S_n^\alpha$ for $\a\le m$ we have
\begin{align*}
D_{1,n}^R(c,f)
&\lesssim \sumjn \frac{j^{2m}f_{j,n}^2}{(j^m+cn)^2}
\lesssim \frac 1{(cn)^2}\sum_{j=1}^{(cn)^{1/m}}j^{2m} f_{j,n}^2+\sum_{j=(cn)^{1/m}+1}^nf_{j,n}^2\\
&\lesssim \frac {(cn)^{(2m-2\a)/m}}{(cn)^2}\sum_{j=1}^{(cn)^{1/m}}j^{2\a} f_{j,n}^2+\frac1{(cn)^{2\a/m}}\sum_{j=(cn)^{1/m}+1}^nj^{2\a}f_{j,n}^2\\
&\le n(cn)^{-2\a/m}.
\end{align*}
In combination with Lemma~\ref{prop:D2nieuw} we find that 
$$\frac 1n D_n^R(c,f)\lesssim (cn)^{-2\a/m}+ n^{-1}(cn)^{1/m}.$$
The argument $c= n^{m/(1+2\a)-1}$ equates the two terms and gives a value of the order $n^{-2\a/(1+2\a)}$.
By Theorem~\ref{thm:deconcluderendeconclusie} this is the
square contraction rate of the plug-in posterior distribution with the risk-based empirical Bayes estimator (\ref{EqRIskEB})
relative to the scaled Euclidean norm $\|\cdot\|_{n,0}$.

For $\a>m$ the order of the square bias $D_{1,n}^R(c,f)$ does not improve beyond the rate $n(cn)^{-2}$ found for
$\a=m$ and hence nor does the contraction rate. 
\end{example}

\begin{example}
[Hyperrectangles]\label{ex:hyperrisk}
Denote by $\Theta_n^\alpha$ the set all functions $f$ for which the discrete Sobolev norm $\|f\|_{n,\a,\infty}$, defined
in (\ref{EqDefFnalpha}), is bounded by 1. 
For eigenvalues satisfying (\ref{EqEigenvalues}) and $f\in \Theta_n^\alpha$ we have
\begin{align*}
D_{1,n}^R(c,f)&\le \sumjn \frac{n j^{-2\a-1}}{(1+c\lambda_{j,n})^2} 
\lesssim n\sumjn \frac{j^{2m-2\alpha-1}}{(j^m+cn)^2} 
\lesssim
\begin{cases} n(cn)^{-2\alpha/m} &\text{ if } \a<m,\\
n(cn)^{-2}\log (cn) &\text{ if } \a=m,\\
n(cn)^{-2}&\text{ if } \a>m.
\end{cases}
\end{align*}
The first case follows directly by Lemma~\ref{lem:som}, the second by writing 
\[
 n\sumjn \frac{j^{2m-2\alpha-1}}{(j^m+cn)^2} = n\sum_{j=1}^{(cn)^{1/m}} \frac{j^{2m-2\alpha-1}}{(j^m+cn)^2} + n\sum_{j=(cn)^{1/m}+1}^n \frac{j^{2m-2\alpha-1}}{(j^m+cn)^2}
\]
and applying a variant of the lemma to the second sum. The third case follows immediately by using $j^m + cn >cn$.
For $\a<m$ and $\a>m$ this is the same result as in  Example~\ref{ExampleSobolev}, leading to the same
conclusions on the contraction rate. For $\a=m$ the additional logarithmic factor leads to the
square contraction rate $n^{-2\a/(2\a+1)}(\log n)^{1/(2\a+1)}$. 
\end{example}

The likelihood-based empirical Bayes method also satisfies an oracle type inequality, but
relative to a loss function that is not as closely linked to the $L_2$-risk of the posterior mean.
Because its ``bias term'' $D_{1,n}^L$ is bigger (the inequality $D_{1,n}^L\ge D_{1,n}^R$ is immediate from
definitions (\ref{EqDandRRisk}) and (\ref{EqDandRLik})),
 while its ``variance term'' $D_{2,n}^L$ has the same order of magnitude, in its attempt to balance
bias and variance the likelihood-based empirical
Bayes method may choose a bigger estimator $\hat c_n$ than the risk-based method.
 This may have an adverse effect on the contraction rate of the plug-in posterior
distribution.

\begin{thm}
[Contraction, likelihood-based EB]
\label{thm:deconcluderendeconclusieLB}Suppose the following conditions hold:
\begin{enumerate}
 \item the remainders $R_{1,n}$ and $R_{2,n}$ behave as in (\ref{EqRemainderR1R2}) and $R_{3,n}$ and $R_{4,n}$ behave as in (\ref{EqRemainderR4}),
 \item $D_{2,n}^L(c)\asymp s_n^2(c)$ uniformly in $c\in\interval$.
\end{enumerate}
Then for $\hat c_n$ given by (\ref{EqLikEB}) and any sequence $M_n\ra\infty$ we have 
\[
 \Pi_{\hat{c_n}}\Bigl(w: \|\vec{w}_n-\vec{f}_n\|\geq M_n \inf_{c\in\interval} D_n^L(c,f)\given \vec Y_n\Bigr) \stackrel{\P_{f}}{\ra} 0.
 \]
In particular, this is true if (\ref{EqEigenvalues}), (\ref{EqEigenvaluesDouble}) or (\ref{EqEigenvaluesDoubleSobolev}) holds.
\end{thm}

\begin{proof}
Since $D_n^L\gtrsim D_n^R$ we obtain as in the proof of Theorem~\ref{thm:deconcluderendeconclusie} that 
$$\|\hat f_{n,\hat c_n}-\vec f_n\|^2=O_P\bigl(D_n^L(\hat c_n,f)\bigr).$$
Next we can use the first assertion of Theorem~\ref{thm:schatter} to replace the right hand side by
the infimum of $D_n^L(c,f)$ over $c$. The posterior variance is of the same order as $D_{2,n}^L$ and
hence the proof can be concluded as the proof of Theorem~\ref{thm:deconcluderendeconclusie}.
\end{proof}

Even though the loss function of the likelihood-based empirical Bayes estimator does 
not relate correctly to the risk in general, the method
does give optimal contraction rates on the models in the preceding examples, albeit for a smaller
range of regularity levels.

\begin{example}
[Sobolev norm]
\label{ExampleSobolev2}
For eigenvalues satisfying (\ref{EqEigenvalues}) and $f\in S_n^\alpha$ for $\a\le m/2$ we have
\begin{align*}
D_{1,n}^L(c,f)
&\lesssim \sumjn \frac{j^{m}f_{j,n}^2}{j^m+cn}
\lesssim \frac 1{cn}\sum_{j=1}^{(cn)^{1/m}}j^{m} f_{j,n}^2+\sum_{j=(cn)^{1/m}+1}^nf_{j,n}^2\\
&\lesssim \frac {(cn)^{(m-2\a)/m}}{cn}\sum_{j=1}^{(cn)^{1/m}}j^{2\a} f_{j,n}^2+\frac1{(cn)^{2\a/m}}\sum_{j=(cn)^{1/m}+1}^nj^{2\a}f_{j,n}^2\\
&\le n(cn)^{-2\a/m}.
\end{align*}
The upper bound has the same form as for the risk-based empirical Bayes method. Since $D_{2,n}^L\asymp D_{2,n}^R$,
we obtain the same contraction rate results. The difference is that the rate does not improve for $\a\ge m/2$.
\end{example}

\begin{example}
[Hyperrectangles]\label{ex:hyperloglik}
\label{ExampleRectangleLB}
For eigenvalues satisfying (\ref{EqEigenvalues}) and $f\in \Theta_n^\alpha$ we have
\begin{align*}
D_{1,n}^L(c,f)&\le \sumjn \frac{n j^{-2\a-1}}{1+c\lambda_{j,n}} 
\lesssim n\sumjn \frac{j^{m-2\alpha-1}}{j^m+cn} 
\lesssim
\begin{cases} n(cn)^{-2\alpha/m} &\text{ if } \a<m/2,\\
c^{-1}\log (cn) &\text{ if } \a=m/2,\\
c^{-1}&\text{ if } \a>m/2,
\end{cases}
\end{align*}
This leads to the contraction rate $n^{-\a/(2\a+1)}$ relative to the scaled Euclidean norm $\|\cdot\|_{n,0}$ if $\a<m/2$
and the square contraction rate $n^{-2\alpha/(2\alpha+1)}(\log n)^{1/(2\alpha+1)}$ if $\a=m/2$.
\end{example}

\subsection{Diameter of the empirical Bayes credible sets}
The empirical Bayes credible sets inherit their diameter from the contraction rate.

\begin{cor}\label{cor:diameteremp}
Under the conditions of Theorems~\ref{thm:deconcluderendeconclusie} and~\ref{thm:deconcluderendeconclusieLB} the 
square of  the diameter  $Mr_n(\hat c_n,\eta)$ of the credible sets (\ref{EqCredibleSetEB}) is of the order  $\inf_{c\in\interval} D_n^R(c,f)$
and $\inf_{c\in\interval} D_n^L(c,f)$ for the risk-based and likelihood-based empirical Bayes procedures respectively with probability tending to one.
\end{cor}

\begin{proof}
By Theorems~\ref{thm:deconcluderendeconclusie} and~\ref{thm:deconcluderendeconclusieLB} 
the empirical Bayes posterior distributions concentrate all their mass on a ball of radius of the same order as the given rate.
Since the posterior distribution is Gaussian, the balls $B_n$ of the same radius centered at the posterior mean must also have mass
tending to one. By definition the credible sets are balls of posterior mass $\eta\in(0,1)$ around the posterior mean, and
hence are contained in the $B_n$.

Alternatively, the square radius $r_n^2(\hat c_n,\eta)$ was seen to be of the same order as the posterior variance $s_n^2(\hat c_n)$, 
which was in turn seen to have the given order. 
\end{proof}

\section{Hierarchical Bayes}
\label{SectionCovHB}

The hierarchical Bayes method is closely related to the likelihood-based empirical Bayes method, since
the posterior density of $c$ is proportional to the product of the 
the prior density $\pi$ for $c$ and the marginal likelihood that defines the latter method. More precisely,
in the model (\ref{EqBayesianModel}) augmented with $c\sim\pi$ it holds that
$$\pi_n(c\given \vec Y_n)\propto p(\vec Y_n\given c)\,\pi(c)
\propto\det\Sigma_{n,c}^{-1/2}\, e^{-\frac12 \vec Y_n^T\Sigma_{n,c}^{-1}\vec Y_n}\,\pi(c).$$
The likelihood-based empirical Bayes estimator (\ref{EqLikEB}) would be the posterior mode if
the prior density were improper. We shall analyse the hierarchical Bayes method by exploiting this link.

We start with showing that the posterior distribution of $c$ concentrates on the interval
where the deterministic part of the likelihood-based criterion $D_n^L=D_{1,n}^L+D_{2,n}^L$ is small. This criterion is derived from minus the log marginal likelihood. On closer inspection it becomes evident that the prior density $\pi$, which we will choose inverse gamma, also plays a role and adds a term $1/c$ to this
criterion. We truncate the inverse gamma prior to the interval $\interval$, so that $c$ has a prior density so that, for some fixed $\k,\lambda>0$,
$$\pi(c)\propto c ^{-1-\k}\, e^{-\lambda/c}, \qquad c\in I_n.$$

\begin{thm}
\label{thm:posterior} Suppose the following conditions hold:
\begin{enumerate}
 \item the remainders $R_{1,n}^L$ and $R_{2,n}^L$ satisfy (\ref{EqRemainderR1R2}),
 \item the function $D_{2,n}^L$ is a good variance function with $D_{2,n}^L(c)\ge \log (nc)$, 
 \item there is a minimizer $c_n(f)$ of $c\mapsto D_n^L(c,f)+2\lambda/c$ over $c\in(0,\infty)$ that satisfies $c_n(f)\in\interval$ and $2c_n(f)\in\interval$. 
\end{enumerate}
Then for sufficiently large $M$
\[
\Pi_n\Bigl(c: D_n^L(c,f)+\frac {1}c \le M \inf_{c>0}\Bigl[D_n^L(c,f)+\frac 1c\Bigr]\given \vec Y_n \Bigr) \stackrel{P_f}{\ra} 1.
\]
Furthermore, if $f$ satisfies the good bias condition relative to $D_{1,n}^L$,
%and  $\sumjn f_{j,n}^2+2b\log n/n\le \sup_{c\in I_n}D_{2,n}^L(c)$,  
then 
\[
 \Pi_n\Bigl(c: D_{1,n}^L(c,f)+\frac1c \lesssim D_{2,n}^L(c)\given \vec Y_n \Bigr) \stackrel{P_f}{\ra} 1.
\]
Moreover, there exist constants $0<k<K<\infty$ such that
\[
 \Pi_n\Bigl(c: c\in \bigl[k c_n(f), Kc_n(f)\bigr]\given \vec Y_n \Bigr) \stackrel{P_f}{\ra} 1.
\]
In particular, these assertions are true if (\ref{EqEigenvalues}),
(\ref{EqEigenvaluesDouble}) or (\ref{EqEigenvaluesDoubleSobolev}) holds, for every $f$ satisfying condition 3.
\end{thm}

\begin{proof}
For every measurable set $J\subseteq I_n$,
\[
\Pi_n\Bigl(c: c\in J \given \vec{Y}_n\Bigr) =\frac{\int_J e^{-\frac12 L_n^L(c,f)}\,\pi(c)\,dc}{ \int_{I_n} e^{-\frac12 L_n^L(c,f)}\,\pi(c)\,dc}
= \frac{\int_J e^{-\frac12 [D_n^L(c,f)+R_n^L(c,f)]}\,\pi(c)\,dc}{ \int_{I_n} e^{-\frac12 [D_n^L(c,f)+R_n^L(c,f)]}\,\pi(c)\,dc},
\]
by the decomposition \eqref{EqDefLn}.
Define $\ell_n(c,f)=D_n^L(c,f)+2\lambda/c$, so that $c_n:=c_n(f)$ is a minimizer of $\ell_n$.
In view of (\ref{EqRemainderR1R2}) we have, for any $\d>0$,
$$\ell_n(c,f)(1-\d)\le D_n^L(c,f)+ R_n^L(c,f)+\frac {2\lambda}c \le \ell_n(c,f)(1+\d),$$
with probability tending to one. Consequently,
\[
\Pi_n\Bigl(c: c\in J \given \vec{Y}_n\Bigr)
\le  \frac{\int_J e^{-\frac12 \ell_n(c,f)(1-\d)}\,c^{-\k-1}\,dc}{ \int_{I_n} e^{-\frac12 \ell_n(c,f)(1+\d)}\,c^{-\k-1}\,dc}.
\]
with probability tending to one. Since $D_{2,n}^L$ is a good variance function, we have that $D_{2,n}^L(2c_n)\le (B')^{-1}2^bD_{2,n}(c_n)$.
Because $D_{1,n}^L$ is decreasing and $D_{2,n}^L$ is increasing, we then also have that
$\ell_n(c,f)\le (B')^{-1}2^b\ell_n(c_n,f)$ for every $c\in [c_n, 2c_n]$. Combining this with the fact that $\ell_n(c,f)\geq 2\lambda/c$, it follows that
\begin{align*}
\Pi_n\Bigl(c:\ell_n(c,f)\ge M \ell_n(c_n,f) \given \vec{Y}_n\Bigr)
&\le  \frac{\int e^{-\frac14 \ell_n(c,f)(1-\d)}\,c^{-\k-1}\,dc \, e^{-\frac 14(1-\d) M\ell_n(c_n,f)}}
{ e^{-\frac1{2B'}2^b(1+\d)\ell_n(c_n,f)}\int_{c_n}^{2c_n} \,c^{-\k-1}\,dc}\\
& \lesssim c_n^\k e^{-\k\ell_n(c_n,f)}\int_0^\infty e^{-\frac12 (1-\d)\lambda/c}\,c^{-\k-1}\,dc
\end{align*}
for $M(1-\d)\ge (4\k+2(B')^{-1}2^b)(1+\d)$. If $c_n\ra0$, then this clearly tends to zero. If $c_n$ is bounded away
from zero, the above also tends to zero, by the assumption that $\ell_n(c,f)\ge \log (cn)$. This concludes the proof of the first assertion of the theorem.

If $f$ satisfies the good bias condition, then, for $K>1$,
$$D_{1,n}^L(Kc,f)+\frac{2\lambda}{Kc}\le K^{-a}D_{1,n}^L(c,f)+\frac{2\lambda}{Kc}
\le K^{-(a\wedge 1)} \Bigl[D_{1,n}^L(c,f)+\frac{2\lambda}{c}\Bigr].$$
In other words, the function $c\mapsto D_{1,n}^L(c,f)+{2\lambda}/{c}$ also satisfies a good bias condition. 

Let $\Lambda_n=\bigl\{c: \ell_n(c,f)\le M\ell_n(\tilde c_n,f)\bigr\}$, for
$\tilde c_n$ the solution to  the equation $D_{1,n}^L(c,f)+2\lambda/c=D_{2,n}^L(c)$. 
%By the assumption on $\sumjn f_{j,n}^2=D_{1,n}^L(0,f)$  and monotonicity of the functions involved, this intersection point exists and does not fall to the right of $\interval$. If it falls to the left, then $D_{1,n}^L\le D_{2,n}^L$ throughout the interval, and the second assertion of the theorem is trivial.
Since $\ell_n(c_n,f)\le \ell_n(\tilde c_n,f)$, we have that $\Pi_n(c: c\in\Lambda_n\given \vec Y_n\bigr)\ra 1$
by the first part of the proof. Since $\ell_n$ is the sum of the decreasing function $D_{1,n}^L(c,f)+2\lambda/c$
and the increasing function $D_{2,n}^L$, which are both ``good'' functions, it follows that
$D_{1,n}^L(c,f)+2\lambda/c\lesssim D_{2,n}^L(c)$ for every $c\in\Lambda_n$ by Lemma~\ref{LemmaCrossing}(i).
Furthermore, Lemma~\ref{LemmaCrossing}(ii) gives the existence of
constants $0<k_1<K_1<\infty$ with $\Lambda_n\subset [k_1\tilde c_n, K_1\tilde c_n]$.
Since $c_n\in\Lambda_n$, it follows that also $\Lambda_n\subset [k_1/K_1 c_n, K_1/k_1 c_n]$.
This proves the second and third assertions of the theorem.
\end{proof}

The theorem shows that under the posterior distribution the scaling $c$ will concentrate
on the set of small values of the criterion $c\mapsto D_{1,n}^L(c,f)+1/c$. This differs by 
the term $1/c$ from the criterion minimized by likelihood-based empirical Bayes estimator
$\hat c_n$ defined by (\ref{EqLikEB}), whose behaviour is given in Theorem~\ref{thm:schatter}.
The additional term is due to the prior distribution. The usual prior distribution, which we consider
here, has very thin tails near 0, and the extra term $1/c$ essentially prevents the posterior
distribution to concentrate very close to zero.

Very small values of the scaling parameter $c$ are
advantageous for very smooth functions $f$. For such functions the bias term $D_{1,n}^L(c,f)$
will be very small and the balance between square bias $D_{1,n}^L(c ,f)$ and variance
$D_{2,n}^L(c)$ will be assumed for small $c$. The additional term can be viewed as adding an artificial bias term 
of the order $1/c$, thus shifting the bias-variance trade-off to bigger values of $c$. 

In most cases this is not harmful. In particular, the shift will not be apparent in contraction rates over the 
usual smoothness models (see Example~\ref{ExampleHBRates}). The following example shows that
this is different for very smooth $f$.

\begin{example}
The smoothest imaginable function $f$ is the zero function. For $f=0$, the bias function
$D_{1,n}^L(c,f)$ in (\ref{EqDandRLik}) vanishes. If the eigenvalues satisfy (\ref{EqEigenvalues}),
then the variance $D_{2,n}^L(c)$ is of the order $(cn)^{1/m}$ by Lemma~\ref{prop:D2nieuw}
and the  criterion  becomes
$$c\mapsto D_n^L(c,f)+\frac{1}c\asymp  (cn)^{1/m}+\frac{1}c.$$
The right side is minimized by $c_n\asymp (1/n)^{1/(m+1)}$.
Theorem~\ref{thm:posterior} shows that the posterior distribution for the scale parameter $c$ will concentrate on  the set
of $c$ that minimize the criterion up to a multiplicative factor. This set is contained in an interval 
with boundaries of the order $(1/n)^{1/(m+1)}$. 

The fact that this interval shrinks to zero is good, as the variance is smaller for smaller $c$, while the bias is negligible. 
However, it is a bit disappointing that the shrinkage is not faster than of order $(1/n)^{1/(m+1)}$. 
In comparison, the empirical Bayes estimator $\hat c_n$ will shrink at the order $\log n/n$,
the minimal possible value permitted in our minimization scheme by Theorem~\ref{thm:schatter}. 
\end{example}

\subsection{Coverage of the hierarchical Bayes credible set}
The hierarchical Bayesian credible sets cover true parameters under the same conditions
as the empirical Bayes sets.

\begin{thm}
[Coverage, HB]
\label{thm:hcoverage}
Suppose the following conditions hold:
\begin{enumerate}
 \item the remainders $R_{1,n}^{L}$ and $R_{2,n}^{L}$ behave as in (\ref{EqRemainderR1R2}) and $R_{3,n}$ and $R_{4,n}$ behave as in (\ref{EqRemainderR4}),
 \item (\ref{EqN}) is satisfied,
 \item $D_{2,n}^L$ is a good variance function with $D_{2,n}^L(c)\ge \log (nc)$,
 \item there is a minimizer $c_n(f)$ of $c\mapsto D_n^L(c,f)+2\lambda/c$ over $c\in(0,\infty)$ that satisfies $c_n(f)\in\interval$ and $2c_n(f)\in\interval$,
 \item $D_{2,n}^R(c)\asymp D_{2,n}^L(c)\asymp s_n^2(c)$ uniformly in $c\in\interval$,
 \item the function $f$ satisfies the good bias condition.
\end{enumerate}
Then the hierarchical Bayes credible sets (\ref{EqCredibleSetHB})
satisfy $P_f(f\in \hat C_{n,\eta,M})\ra 1$ for sufficiently large $M$. In particular, this is true if (\ref{EqEigenvalues}),
(\ref{EqEigenvaluesDouble}) or (\ref{EqEigenvaluesDoubleSobolev}) holds and conditions 4 and  6 hold.
%and $\sumjn f_{j,n}^2\le \sup_{c\in I_n}D_{2,n}(c)$.
\end{thm}

\begin{proof}
The function $f$ is contained in $\hat C_{n,\eta,M}$ as soon as there exists some $c\in [\hat c_{1,n}(\eta_1),\hat c_{2,n}(\eta_1)]$
with $\|\vec f_n-\hat f_{n,c}\|\le M r_n(c,\eta_2)$.  Since $N_n(c)/s_n^2(c)\ra 1$ in probability uniformly in $c\in \interval$ by (\ref{EqN}),
the quantities $r_n^2(c,\eta_2)/s_n^2(c)$, which are the $\eta_2$-quantiles of the variables $N_n(c)/s_n^2(c)$,
tend to 1 as well uniformly in $c$. In view of the decomposition (\ref{EqDecompositionSquareNorm}) it follows that 
the function $f$ is contained in $\hat C_{n,\eta,M}$ as soon as there exists some $c\in [\hat c_{1,n}(\eta_1),\hat c_{2,n}(\eta_1)]$ with
$$\frac{D_{1,n}^{R}(c,f)+D_{2,n}^{R}(c)+R_{3,n}(c,f)+R_{4,n}(c)}{s_n^2(c)} \le M^2\bigl(1+o_P(1)\bigr).$$
By assumption $s_n^2(c)$  is equivalent
to both $D_{2,n}^R(c)$ and $D_{2,n}^L(c)$, up to a multiplicative constant. In particular,
the second term on the left is bounded above. 

By the second assertion of Theorem~\ref{thm:posterior} 
the posterior probability of the set $\Lambda_n:=\bigl\{ c: D_{1,n}^L(c,f) \lesssim D_{2,n}^L(c)\bigr\}$  tends to one in probability.
Since $\hat c_{1,n}(\eta_1)$ and $\hat c_{2,n}(\eta_1)$ are nontrivial quantiles of the posterior distribution of $c$, 
the interval $[\hat c_{1,n}(\eta_1),\hat c_{2,n}(\eta_1)]$ must intersect $\Lambda_n$ with probability tending to 1.
For $c=\bar c_n$ in this intersection it holds that $D_n^L(c,f)\asymp D_{2,n}^L(c)$ and hence $s_n^2(c)$ in the preceding display
can be replaced by $D_n^L(c,f)$, up to a multiplicative constant. This shows that the remainder terms tend to zero,
in view of (\ref{EqRemainderR4}). The first term $D_{1,n}^R(c,f)/s_n^2(c)$ is bounded by a multiple of $D_{1,n}^R(c,f)/D_n^L(c,f)\le D_{1,n}^R(c,f)/D_{1,n}^L(c,f)\le1$,
by definitions (\ref{EqDandRRisk}) and (\ref{EqDandRLik}). This proves the first assertion of the theorem.

The final assertion of the theorem follows by Propositions~\ref{prop:R},~\ref{prop:R4} and~\ref{prop:postunif}
and Lemma~\ref{prop:D2nieuw}, which show that all remaining assumptions hold
under (\ref{EqEigenvalues}), (\ref{EqEigenvaluesDouble}) or (\ref{EqEigenvaluesDoubleSobolev}).
% and the condition $\sumjn f_{j,n}^2\le \sup_{c\in I_n}D_{2,n}(c)$.
\end{proof}

\subsection{Contraction rate of of the hierarchical Bayes posterior}
As in Section~\ref{sec:contraction} write $\Pi_c\bigl(\cdot\given \vec Y_n\bigr)$ for the posterior 
distribution of $\vec f_n$ given $c$. Then the hierarchical posterior distribution can be decomposed as
$$\Pi_n\bigl(w: \vec w_n\in B \given \vec Y_n\bigr) 
= \int \Pi_c\bigl(w: \vec{w}_n\in B\given \vec Y_n\bigr)\,\pi_n(c \given \vec Y_n)\,dc$$
for $B\subseteq \mathbb{R}^n$ measurable. 
Here  $\pi_n(c\given \vec Y_n)$ is the posterior density of $c$, analysed in Theorem~\ref{thm:posterior}.

This hierarchical posterior distribution contracts to the true parameter according to an oracle inequality,
with the likelihood-based criterion augmented by the extra term $1/c$.

\begin{thm}
[Contraction rate, HB]
\label{thm:contractionHB}
If conditions 1, 3, 4, and 5 of Theorem~\ref{thm:hcoverage} hold, then,
for any sequence $M_n\ra\infty$,
\[
\Pi_n\Bigl(w:\|\vec{w}_n-\vec{f}_n\|^2\geq M_n\inf_{c\in \interval}\Bigl[D_n^L(c,f)+\frac 1c\Bigr] \given \vec Y_n\Bigr) \stackrel{\P_f}{\ra} 0.
 \]
\end{thm}

\begin{proof} 
Let $c_n\in \interval$ be a minimizer of $c\mapsto D_n^L(c,f)+ 1/c$ and for given $M_1$ define a set 
\begin{equation}
\label{EqDefCn}
C_n=\Bigl\{c\in\interval: D_n^L(c,f)+ 1/c \le M_1\bigl[D_n^L(c_n,f)+ 1/c_n\bigr]\Bigr\}.
\end{equation}
By Theorem~\ref{thm:posterior} the posterior probability that $c\in C_n$ tends to 1 in probability, for sufficiently large $M_1$.
Therefore, for any $M>0$ we apply the above decomposition of the posterior to find
\begin{align*}
\Pi_n\bigl(w: \|\vec{w}_n-\vec{f}_n\|\geq M \given \vec Y_n\bigr) 
&\le \sup_{c\in C_n}\Pi_c\bigl(w: \|\vec{w}_n-\vec{f}_n\|\geq M\given  \vec Y_n\bigr) + \Pi_n(c: c\notin C_n\given \vec Y_n)\\
&\le \frac 1{M^2}\sup_{c\in C_n}\bigl[ \|\hat f_{n,c}-\vec f_n\|^2+s_n^2(c)\bigr]+o_P(1)
\end{align*}
by Markov's inequality. In view of (\ref{EqDecompositionSquareNorm}), this is further
 bounded above by 
$$\frac1{M^2}\sup_{c\in C_n}\Bigl[ D_{1,n}^R(c,f)+D_{2,n}^R(c)+R_{3,n}(c,f)+R_{4,n}(c)+s_n^2(c)\Bigr]+o_P(1).$$
Here $D_{1,n}^R\le D_{1,n}^L$, and $D^R_{2,n}$ is of the same order as  $D_{2,n}^L$ and $s_n^2$. It follows that
the first two terms are bounded by a multiple of $\sup_{c\in C_n}D_n^L(c,f)\le M_1 D_n^L(c_n)+1/c_n$. The remainder terms are
of the order $D_n^L(c,f)$ uniformly in $c\in\interval$ with probability tending to one by (\ref{EqRemainderR4}) and hence
are similarly bounded.
\end{proof}

\begin{example}
[Sobolev]
\label{ExampleHBRates}
It was seen in Example~\ref {ExampleSobolev2} 
that for eigenvalues satisfying (\ref{EqEigenvalues}) and $f\in S_n^\alpha$ for $\a\le m/2$ we have
\begin{align*}
D_{1,n}^L(c,f)+D_{2,n}^L(c)&\lesssim n(cn)^{-2\a/m}+(cn)^{1/m}.
\end{align*}
The upper bound on the right side has minimum value $n^{1/(2\a+1)}$ at $c_n\asymp n^{m/(1+2\a)-1}$. 
In this point the term $1/c_n$ is smaller than $n^{1/(2\a+1)}$ (for $\a\le m/2$). It follows from
Theorem~\ref{thm:contractionHB} that on the model $S_n^\alpha$ the 
hierarchical Bayes posterior distribution contracts at the same rate as the likelihood-based empirical Bayes method. 
\end{example}

\begin{example}
[Hyperrectangle]
It was seen in Example~\ref{ExampleRectangleLB}
that, for eigenvalues satisfying (\ref{EqEigenvalues}) and $f\in \Theta_n^\alpha$, 
\begin{align*}
D_{1,n}^L(c,f)+D_{2,n}^L(c)&\lesssim 
\begin{cases} n(cn)^{-2\alpha/m}+(cn)^{1/m} &\text{ if } \a<m/2,\\
c^{-1}\log (cn) +(cn)^{1/m}&\text{ if } \a=m/2,\\
c^{-1}+(cn)^{1/m}&\text{ if } \a>m/2,
\end{cases}
\end{align*}
It follows again that the 
hierarchical Bayes posterior distribution contracts at the same rate as the likelihood-based empirical Bayes method. 
\end{example}

\begin{example}
[Zero function]
The square bias $D_{1,n}^L$ of the function $f=0$ is equal to zero. For eigenvalues satisfying (\ref{EqEigenvalues}) the minimum of 
$c\mapsto D_n^L(c,f)+1/c$ is assumed at $c_n\asymp (1/n)^{1/(m+1)}$, resulting in a rate of contraction for 
the scaled Euclidean norm $\|\cdot\|_{n,0}$ of the order $n^{-(m/2)/(m+1)}$.

In contrast the empirical Bayes estimators attain a rate of contraction of the order $n^{-1/2}$ up to a logarithmic factor.

The same difference between the hierarchical and empirical Bayes
methods exists for (sequences of) functions $f$ with a square bias $D_{1,n}^R(c,f)$ that
tends to zero at an exponential rate. 
%However, fixed functions with this property are rare.
\end{example}

\subsection{Diameter of the hierarchical Bayes credible set}
The diameter of the credible sets is again of the same order as the contraction rate.

\begin{thm}
\label{thm:width} 
Under the conditions of Theorem~\ref{thm:contractionHB} the diameter of the
credible sets (\ref{EqCredibleSetHB}) is of the order $\inf_{c\in I_n} \bigl[D_n^L(c,f)+1/c\bigr]$ with probability tending to one.
\end{thm}

\begin{proof}
In view of Proposition~\ref{prop:postunif}, for fixed $c$ the radius of the credible set $\{w: \|\vec w_n-\hat f_{n,c}\|< Mr_n(c,\eta_2)\}$ is
of the order the posterior standard deviation $s_n(c)$ given by (\ref{EqPosteriorVariance}).
Thus the triangle inequality gives that the diameter of $\hat C_{n,\eta,M}$ is bounded above by a multiple of 
$$\sup_{\hat c_{1,n}(\eta_1)< c< \hat c_{2,n}(\eta_1)}\bigl[s_n(c)+ \|\vec f_n-\hat f_{n,c}\|\bigr].$$
The supremum of the function in this display over the set $C_n$  defined in
(\ref{EqDefCn}) is shown to be of the desired order in the proof of Theorem~\ref{thm:contractionHB}. 
The theorem would follow if the interval $[\hat c_{1,n}(\eta_1), \hat c_{2,n}(\eta_1)]$ belongs to $C_n$
 with probability tending to one.

By Theorem~\ref{thm:posterior} the posterior distribution of $c$ concentrates all its mass on the sets $C_n$. 
Since $\hat c_{1,n}(\eta_1)$ and $\hat c_{2,n}(\eta_1)$ are nontrivial quantiles of this
distribution, we can conclude that they must belong to the convex hull of $C_n$ with probability tending to one.
If this convex hull is $[c_m,c_M]$, then for any $c$ in this convex hull
$$D_n^L(c,f)+\frac1c= D_{1,n}^L(c,f)+\frac1c+D_{2,n}^L(c)\le D_{1,n}^L(c_m,f)+\frac1{c_m}+D_{2,n}^L(c_M)
\le 2M_1\Bigl[ D_n^L(c_n,f)+\frac1{c_n}\Bigr].$$
Thus the convex hull of $C_n$ is contained in a set of the same form as $C_n$, but with the constant
$M_1$ replaced by $2M_1$. The proof of Theorem~\ref{thm:contractionHB} still shows that
the supremum over this bigger set is of the desired order.
\end{proof}

\section{On the polished tail condition}
\label{SectionPolishedTail}
The parameter in the regression model (\ref{eq:RegProblem}) is a fixed function $f$,
but most of the results of this paper are driven by the representation of the restriction
$\vec f_n$ of $f$ to the design points in terms of the eigenvectors $e_{j,n}$ of the covariance matrix
$U_n$ of the (unscaled) prior restricted to the design points. It is clearly of interest to relate the ``continuous''
object $f$ to its discrete counterparts, but this is more involved than it may seem. 

In this section we investigate the relationship between the continuous and discrete setups
for the special case of the Brownian motion prior.

\subsection{Aliasing}
\label{SectionAliasing}
For the design points $x_{i,n}=i/n_+$, where $n_+=n+1/2$, the eigenvectors of the covariance matrix $U_n$ of 
discretized Brownian motion are given in (\ref{EqEigenvectors}) for $j\in\{1,\ldots, n\}$. The formula
shows that they are $1/\sqrt{n+}$ times the restrictions of the eigenfunctions $e_{j}$ to the design points.
Using this correspondence we may also define vectors $e_{j,n}\in \RR^n$  for $j>n$, again by (\ref{EqEigenvectors}), by discretizing the 
higher frequency eigenfunctions of Brownian motion. Since the vectors $e_{1,n},\ldots, e_{n,n}$ are
an orthonormal basis of $\RR^n$, these further vectors are redundant. It turns out that 
their linear dependency on the vectors $e_{i,n}$ for $i\le n$ takes a very special form:
\begin{enumerate}
\item[(i)] The vectors $e_{i,n}$ are $(2n+1)$-periodic in $i$: $e_{i+2n+1,n}=e_{i,n}$ for
  all $i$.
\item[(ii)] The vectors in the middle of a $(2n+1)$ period vanish: $e_{n+1,n}=0$.
\item[(iii)] The vectors within a $(2n+1)$ period are anti-symmetric about the midpoint: $e_{2n+2-i,n}=-e_{i,n}$ for all $i$.
\end{enumerate}
In particular, every $e_{j,n}$ with $j>n$ is either zero or ``loads'' on exactly one $e_{i,n}$ with $i\in\{1,\ldots,n\}$
with coefficient 1 or -1. This leads to a simple connection between 
the infinite expansion of a function $f=\sum_{j=1}^\infty f_je_j$ in the eigenfunctions $e_j$ of continuous Brownian motion
and the finite expansion $\vec f_n=\sumin f_{i,n} e_{i,n}$ of the discretized function
$\vec f_n$  in the eigenvectors $e_{j,n}$ of discretized Brownian motion, as follows. Assuming that
the series $f(x)=\sum_{j=1}^\infty f_je_j(x)$ converges pointwise, we can use 
(\ref{EqEigenvectors}), which says that $(\vec e_j)_n=\sqrt{n_+} e_{j,n}$, and (i)-(iii) to see that the coefficients in $\vec{f}_n$ are given by
\begin{equation}\label{eq:Aliasing}
f_{i,n}=\sum_{j=0}^\infty f_j (\vec{e_j})_n^Te_{i,n} = \sqrt{n_+}\,\sum_{l=0}^\infty(f_{(2n+1)l+i}-f_{(2n+1)l +2n+2-i}).
\end{equation}
The terms of this last series correspond to the consecutive periods of lengths $(2n+1)$. Exactly two of the inner products per period are nonzero and they yield coefficients $1$ and $-1$ respectively. The formula is an example of the \emph{aliasing} effect in signal analysis: the energy of the
function $f$ at frequencies $j$ higher than the Nyquist frequency $n$, whose fluctuations fall between the grid points, is represented at the lower frequencies.
 
The scaling $\sqrt{n+}$ results from the normalisation of the vectors $e_{i,n}$ in $\RR^n$.
However, even apart from the normalisation the correspondence between the discrete and continuous
coefficients is imperfect. By writing (\ref{eq:Aliasing}) in the form
$$\frac{f_{i,n}}{\sqrt{n_+}} =f_i-f_{2n+2-i}+\sum_{l=1}^\infty(f_{(2n+1)l+i}-f_{(2n+1)l +2n+2-i}),$$
we see that $f_{i,n}/\sqrt{n_+}$ is in general not equal to $f_i$. The ``harmonic frequencies'' at periods
$2n+1$ add to a frequency at $i\in\{1,2,\ldots,n\}$, and the frequencies mirrored around the
midpoints of the blocks subtract from it. 

It is clear from the preceding display that a given discrete sequence $(f_{i,n})$ can be obtained from
the infinite sequence $(f_{1,n}, f_{2,n},\ldots, f_{n,n},0,0,\ldots)/\sqrt{n_+}$ of $L^2$ coefficients, but also from many other
infinite sequences $(f_j)$. Because the data model (\ref{eq:RegProblem}) depends on $f$ only through
the discrete sequence $(f_{i,n})$, there is clearly no hope to recover which of these infinite sequences
would be the ``true'' sequence.  Furthermore, for a given fixed infinite sequence the values of the array $(f_{i,n})$ 
will change with $n$, and for some reasonable infinite sequences the series defining the discrete coefficients
may not even converge. (We obtained the preceding display under the assumption that the series $\sum_jf_je_j(x)$
converges pointwise.) The following lemma shows that the infinite series is essentially a Fourier series, 
and hence this less than perfect correspondence is disappointing.

\begin{lem}\label{lem:LemmaOrdinaryFourier}
For a given $f: [0,1]\to\RR$ in $L_2[0,1]$, the expansion $f=\sum_j f_j e_j$ is derived from the Fourier series of the function
$x\mapsto e^{i\pi x/2}f(x)$ on $[0,2]$, where $f$ is extended to $[0,2]$ by symmetry about 1. 
In particular, if $f\in C^\alpha[0,1]$ for some $\a>0$ and $f(0)=0$, then
\[
f(x)=\sum_{j=1}^\infty f_j e_j(x),\qquad \text{ uniformly in } x.
\]
\end{lem}

\begin{proof}
The function $x \mapsto e^{i\pi x/2}f(x)$, with $f$ extended as indicated,
is periodic (i.e.\ it has the same value at $0$ and $2$) and contained in $L_2[0,2]$. Its Fourier series can be written in the form
\begin{equation}\label{EqFourier}
 e^{i\pi x/2}f(x)= \sum_{j\in \mathbb{Z}}c_j e^{i\pi jx}
\end{equation}
for some $c_j\in \mathbb{C}$ and hence
\[
 f(x)=\sum_{j\in \mathbb{Z}}c_j e^{i\pi (j-\frac{1}{2})x}.
\]
Since $f$ is real, the complex part of the right side vanishes, while the real part can be written in the form
\[
 f(x)=\sum_{j\in \mathbb{Z}} a_j \cos\bigl(\pi x(j-1/2)\bigr)-b_j\sin \bigl((j-1/2)\pi x\bigr),
\]
for $a_j, b_j\in\mathbb{R}$. Since $f$ is symmetric about 1,  the antisymmetric cosine part vanishes, 
while the terms with $j\le 0$ of the sine part can be united with terms with $j\ge 1$. This gives an expansion in terms
of the eigenfunctions $e_j$. By the orthogonality of these functions the resulting expansion is unique.

If $f \in C^\alpha[0,1]$, then the extended function $x\mapsto e^{i\pi x/2}f(x)$ is contained in $C^\alpha[0,2]$
and hence we uniform convergence in \eqref{EqFourier}. The uniform convergence is retained under multiplying left and right with $e^{-i\pi x/2}$.
\end{proof}

As a consequence of the lemma, the speed at which the $f_j$ tend to zero as $j\ra\infty$ can be interpreted in the
sense of Sobolev smoothness. However, this is not easily comparable to the smoothness of the corresponding array
$(f_{i,n})$. In fact, if $f$ is contained in a Sobolev space of order $\a$ for $\a\le 1/2$, that is $\sum_j j^{2\a}f_j^2<\infty$, then
the aliased coefficients may not even be well defined.

\subsection{Polished tail sequences}
In \cite{Szabo} a function $f$, or rather its infinite series of coefficients $(f_j)$ relative to a given eigenbasis, is defined to be \emph{polished tail}
if for some $L, \rho>0$ and all sufficiently large $m$,
\begin{equation}
\label{EqInfinitePolishedTail}
\sum_{j=m}^\infty f_j^2\le L \sum_{j=m}^{\rho m} f_j^2.
\end{equation}
This reduces to the ``discrete polished tail'' condition (\ref{eq:EqPolishedTail}) if applied to the infinite sequences
$(f_{1,n}, f_{2,n},\ldots, f_{n,n},0,0,\ldots)/\sqrt{n_+}$. For general sequences $(f_j)$ the relationship
is less perfect, but for typical examples the two concepts agree.

\begin{example}[Self-similar sequences]
\label{ex:aliasedpt}
In \cite{Szabo} an infinite sequence $(f_j)$ is defined to be \emph{self-similar} of order $\a>0$ if
for some positive constants $M, \rho, L$ and every $m$,
$$\sup_{j\ge 1}j^{1/2+\a}|f_j|\le M,\qquad\text{ and }\qquad  \sum_{j=m}^{\rho m} f_j^2\ge M^2L m^{-2\a}.$$
Particular examples are the sequences with the exact order $|f_j|\asymp j^{-1/2-\alpha}$.
Self-similar sequences are easily seen to be polished tail for every $\alpha>0$ and arbitrary $\rho>1$. 
For $\alpha\le 1/2$ the corresponding function is not necessarily well defined
at every point and the series \eqref{eq:Aliasing} defining the aliased coefficients may
diverge. However, for $\alpha>1/2$ the induced array $(f_{i,n})$ is well defined and also discrete
polished tail in the sense of (\ref{eq:EqPolishedTail}).

To see this, first note that for $\ell\ge 1$ and taking $M$ equal to 1 for simplicity we have
\[
 |f_{(2n+1)\ell+i}|\vee|f_{(2n+1)\ell+2n+2-i}|\lesssim \frac {1}{n^{1/2+\alpha}\ell^{1/2+\alpha}}.
\]
This shows that the series \eqref{eq:Aliasing} that defines the aliased coefficients 
converges. Furthermore, we see that the rescaled coefficients
$\tilde f_{i,n}=f_{i,n}/\sqrt{n_+}$ satisfy $|\tilde f_{i,n}-f_i|\lesssim  n^{-1/2-\alpha}$, so that
$|\tilde f_{i,n}|\lesssim i^{-1/2-\alpha}+n^{-1/2-\alpha}$ and
the left side of \eqref{eq:EqPolishedTail} satisfies
\[
\sum_{i=m}^n \tilde f_{i,n}^2 \lesssim \frac1{m^{2\a}}+\frac1{n^{2\alpha}}\lesssim \frac1{m^{2\alpha}}.
\]
We wish to show that the right side of \eqref{eq:EqPolishedTail} is lower bounded by the expression on the right,
where we may assume that $m$ satisfies $\rho m\le n$, because otherwise there is nothing to prove. 
First we note that 
$$|\tilde f_{i,n}^2-f_i^2|=|\tilde f_{i,n}-f_i|\,|\tilde f_{i,n}+f_i|\lesssim
\frac1{n^{1/2+\alpha}}\Bigl(|f_i|+\frac1{n^{1/2+\a}}\Bigr).$$
It follows that, for some universal constant $C$,
\[\sum_{i=m}^{\rho m\wedge n} \tilde f_{i,n}^2
\ge\sum_{i=m}^{\rho m} f_i^2 -\frac{C(\rho-1)m}{n^{1+2\a}}- C\sum_{i=m}^{\rho m} \frac{|f_i|}{n^{1/2+\a}}
\gtrsim \frac1{m^{2\alpha}}\Bigl(L-\frac{2C(\rho-1)}{\rho^{1/2+\a}}\Bigr).
\]
For sufficiently large $L$ the constant in the last display is positive.
\end{example}

\begin{example}
The sequence $f_j=j^{-1/2-\alpha}$ is easily seen to be polished tail for every $\alpha>0$, as is also
noted in Example~\ref{ex:aliasedpt}. We shall show that the corresponding array $(f_{i,n})$ is also discrete polished tail
in the sense of (\ref{eq:EqPolishedTail}), for any $\a>0$, thus extending Example~\ref{ex:aliasedpt} to the range $\a\in (0,1/2]$.
This refinement is possible by the exact form of the $f_j$, which allows us to exploit cancellation of positive and negative terms
in \eqref{eq:Aliasing}.

To prove the claim we first apply the mean value theorem to find that,  for every $\ell\ge 1$, 
$$|f_{(2n+1)\ell+i}-f_{(2n+1)\ell+2n+2-i}| %= \Bigl|\frac1{\bigl((2n+1)\ell+i\bigr)^{1/2+\alpha}} -\frac1{\bigl((2n+1)\ell+2n+2-i \bigr)^{1/2+\alpha}} \Bigr| 
\lesssim \frac1{n^{1/2+\alpha}\ell^{3/2+\alpha}}.$$
This shows that the series in \eqref{eq:Aliasing} defining the discrete coefficients converges. Moreover, 
$$|\tilde f_{i,n}| \lesssim \frac2{i^{1/2+\alpha}} + \sum_{\ell=1}^\infty |f_{(2n+1)\ell+i}-f_{(2n+1)\ell+2n+2-i}| 
\lesssim \frac1{i^{1/2+\alpha}} +\frac1{n^{1/2+\alpha}}.$$
Consequently, the left side of \eqref{eq:EqPolishedTail} satisfies
$$\sum_{i=m}^n \tilde f_{i,n}^2\lesssim \frac 1{m^{2\alpha}}+\frac1{n^{2\alpha}}\lesssim \frac1{m^{2\alpha}}.$$
Furthermore, since all terms in \eqref{eq:Aliasing} are positive, we also have
$$\tilde f_{i,n}\ge \frac1 {i^{1/2+\alpha}}-\frac 1{(2n+2-i)^{1/2+\alpha}}\gtrsim \frac1 {i^{1/2+\alpha}},$$ for $i\le cn$ and any fixed $c<1$. 
To bound the right side of \eqref{eq:EqPolishedTail} we may assume that $m$ satisfies $\rho m\le n$,
because otherwise there is nothing to prove. Then choosing $c<1$ and $\rho>1$ such that $c\rho>1$,
we have
\[
\sum_{i=m}^{\rho m\wedge n} \tilde f_{i,n}^2
\ge \sum_{i=m}^{c\rho m} \tilde f_{i,n}^2 
\gtrsim \sum_{i=m}^{c\rho m} \frac1 {i^{1+2\alpha}}\geq \int_m^{c\rho m} \frac{1}{t^{1+2\alpha}} \intd t \gtrsim \frac1{m^{2\alpha}}.
\]
The right side is seen to be bigger than a multiple of the left side of \eqref{eq:EqPolishedTail}. This proves the claim.
\end{example}

\subsection{Prior polished tail sequences}
According to the Bayesian model the true function $f$ is a realisation of the prior process $W^c$. 
In this section we show that almost every such realisation gives rise to a discrete polished tail array.
Consequently, for a Bayesian who believes in her prior, the polished tail condition is reasonable. For a non-Bayesian 
the following proposition is also of interest, as it shows that polished tail functions are abundant.

The proof of the statement will be based  on the Karhunen-Lo\`eve expansion. For standard Brownian
motion $W^1=(W_t^1: t\in [0,1])$ this is given by
\[
 W_t^1=\sum_{j=1}^\infty \frac{Z_j}{(j-1/2)\pi} e_j(t).
\]
Here $Z_1,Z_2,\ldots$ are independent standard normal random variables. We see that the prior $W^c$ is given by $\sum_j f_je_j$, for the infinite sequence $f_j=\sqrt cZ_j/((j-1/2)\pi)$. We shall show that the induced array
$f_{j,n}$ defined by (\ref{eq:Aliasing}) is discrete polished tail, almost surely.

In fact a more general result holds for any Gaussian series with polynomially
decaying singular values relative to the eigenbasis of Brownian motion.

\begin{prop}\label{prop:priorpolished}
For given $\alpha>0$ and $\delta\in\mathbb{R}$ set
\[
 W_t=\sum_{j=1}^\infty \frac{Z_j}{(j+\delta)^{1/2+\alpha}}e_j(t),\qquad t\in [0,1],
\]
where $Z_1,Z_2,\ldots$ are independent standard normal random variables. Then almost every realisation of $W$ is both polished tail in the sense of (\ref{EqInfinitePolishedTail}) and discrete polished tail in the sense of (\ref{eq:EqPolishedTail}).
\end{prop}

\begin{proof}
The first claim is proved in Proposition~3.5 of \cite{Szabo}. To prove that $W$ is discrete polished
tail, we consider the coefficients given in (\ref{eq:Aliasing}):
\[
 W_{i,n} = \sum_{l=0}^\infty \left( \frac{Z_{(2n+1)l +i}}{(\delta+(2n+1)l+i)^{1/2+\alpha}} - \frac{Z_{(2n+1)l +2n+2-i}}{(\delta + (2n+1)l +2n+2-i)^{1/2+\alpha}}\right).
\]
In view of L\'evy's continuity theorem this array consists for each $n$ of independent zero-mean
normal random variables  $W_{1,n}, W_{2,n},\ldots, W_{n,n}$ with variances
\begin{equation}
\label{eq:EqVarAliased}
\var\bigl(W_{i,n}\bigr) \asymp \sum_{l=0}^\infty \left(\frac{1}{((2n+1)l+i)^{2\alpha+1}}+\frac{1}{((2n+1)l+2n+2-i)^{2\alpha+1}}\right).
\end{equation}
Now let $L,\rho >0$ and consider the event $E_m = \bigl\{ \sum_{i=m}^n W^2_{i,n} > L\sum_{i=m}^{\rho m} W^2_{i,n} \bigr\}$. Setting
\[
 X =L \sum_{i=m}^{\rho m} W^2_{i,n} - \sum_{i=m}^n W^2_{i,n} = (L-1)\sum_{i=m}^{\rho m} W^2_{i,n} - \sum_{i=\rho m + 1}^n W^2_{i,n},
\]
we see that $E_m$ has probability $\P(E_m)=\P(X<0)$. We then have by Markov's inequality that for $\eta>0$
\[
 \P(E_m) = \P\left(X < 0\right) \leq \P \left( |X-\E X| \geq \E X \right)\leq \frac{\E |X-\E X|^\eta}{(\E X)^\eta}.
\]
We proceed to bound the expectation of $X$. Clearly the right hand side of \eqref{eq:EqVarAliased} is bigger than $i^{-1-2\alpha}$. Since $i\le n$, it is also smaller than
\begin{align*}
&\frac 1{i^{2\alpha+1}}+\frac 3{(2n+1)^{2\alpha+1}}+2\int_1^\infty \frac 1{((2n+1)x+i)^{2\alpha+1}}\intd x\\
&\quad\le \frac 1{i^{2\alpha+1}}+\frac 3{(2n+1)^{2\alpha+1}}+2\int_1^\infty \frac 1{((2n+1)x+i)^{2\alpha+1}}\intd x
\le  \frac 1{i^{2\alpha+1}}+L_1\frac 1{n^{2\alpha+1}},
\end{align*}
for some $L_1>0$. It follows that
\begin{align*}
\E X %&= (L-1)\sum_{i=m}^{\rho m}\E W_{i,n}^2-\sum_{i=\rho m + 1}^n \E W_{i,n}^2 \\
&\ge (L-1)\sum_{i=m}^{\rho m}\frac 1{i^{2\alpha+1}}-\sum_{i=\rho m + 1}^n \frac 1{i^{2\alpha+1}}-L_1 \sum_{i=\rho m + 1}^n \frac 1{n^{2\alpha+1}}\\
&\ge \frac{1}{2\alpha}\frac1{m^{2\alpha}}\bigl[(L-1)(1-\rho^{-2\alpha})-(1+L_1)\rho^{-2\alpha}\bigr].
\end{align*}
We choose $L$ and $\rho$ large enough so that this is positive. Applying the Marcinkiewicz-Zygmund inequality and next H\"older's inequality with conjugate parameters $(\eta/2,\eta/(\eta-2))$, we obtain for $\eta>2$: %Note valid for infinite sums by monotone convergence theorem
\begin{align*}
\E|X-\E X|^\eta &\lesssim \E \left(\sum_{i=m}^{\rho m}(L-1)^2 \bigl(W_{i,n}^2-\E W_{i,n}^2\bigr)^2 +\sum_{i=\rho m+1}^{n} \bigl(W_{i,n}^2-\E W_{i,n}^2\bigr)^2\right)^{\eta/2}\\
&\lesssim \E \left( \left(\sum_{i=m}^n |W_{i,n}^2-\E W_{i,n}^2|^\eta i^{\eta/2}\right)^{2/\eta}  \left( \sum_{i=m}^n i^{-\eta/(\eta-2)} \right)^{1-2/\eta}\right)^{\eta/2}\\
&=\sum_{i=m}^n \E |W_{i,n}^2-\E W_{i,n}^2|^\eta i^{\eta/2} \left( \sum_{i=m}^n i^{-\eta/(\eta-2)} \right)^{\eta/2-1}.
\end{align*}
Since $\E|W_{i,n}^2-\E W_{i,n}^2|^\eta \asymp \var(W_{i,n})^{\eta}\lesssim i^{-(1+2\alpha)\eta}$, we conclude
\begin{align*}
 \E |X-\E X|^\eta &\lesssim \sum_{i=m}^n i^{\left(1/2 - (1+2\alpha)\right)\eta} \left( \sum_{i=m}^n i^{-\eta/(\eta-2)} \right)^{\eta/2-1}\\
 &\lesssim m^{1-(1/2+2\alpha)\eta +\eta/2-1 - \eta/2}=m^{-(1/2+2\alpha)\eta},
\end{align*}
hence the $\P(E_m)$ are bounded by a multiple of $m^{-\eta/2}$ and thus summable over $m$ for $\eta>2$. It follows by the Borel-Cantelli lemma that the event $E_m$ occurs at most finitely many times, with
probability one.
\end{proof}

\section{Discussion} 
\label{SectionDiscussion}
The model (\ref{eq:RegProblem}) can also be formulated directly in terms of the coordinates $(f_{i,n})$ of
$\vec{f}_n$ relative to the eigenbasis $e_{j,n}$ of the prior covariance matrix $U_n$. 
For $O_n$ the orthogonal matrix with rows the eigenvectors $e_{j,n}$ of $U_n$, 
the definition of $f_{j,n}$ gives
$$%\left(\begin{matrix} \tilde Y_{1,n}\\\tilde Y_{2,n}\\ \vdots\\\tilde Y_{n,n}\end{matrix}\right)
O_n\vec Y_n=O_n \vec f_n+O_n\vec\e_n=\left(\begin{matrix} f_{1,n}\\f_{2,n}\\ \vdots\\f_{n,n}\end{matrix}\right)+O_n\vec\e_n.$$
By the orthonormality of $O_n$ the error vector $O_n\vec\e_n$ is equal in distribution to $\vec\e_n$,
whence $\tilde Y_n=O_n\vec Y_n$ can be considered a vector of observations in a normal mean model with mean vector $(f_{i,n})$. Under the prior $W^c$ on $f$, given $c$ the vector $(f_{1,n},\ldots, f_{n,n})^T=O_n^{-1}\vec f_n$
possesses a mean zero normal distribution with covariance matrix 
$cO_n^{-1}U_nO_n=\diag(c\lambda_{i,n})$. Prior and data model both factorise over the coordinates,
and it can be seen that under the posterior distribution given $c$ the variables
$f_{1,n},\ldots, f_{n,n}$ are again independent with 
$$f_{i,n}\given \vec Y_n,c \sim \mathcal{N}\left( \frac{c\l_{i,n}}{1+c\l_{i,n}}\tilde Y_{i,n}, \frac{c\l_{i,n}}{1+c\l_{i,n}}\right).$$
This gives a representation of the posterior distribution different from, but equivalent  to (\ref{EqPosterior}).

In this form the model resembles the infinite Gaussian sequence model (or white noise model).
A difference is that presently the sequence is of length $n$ instead of infinite, and the parameter vector
$(f_{1,n},\ldots, f_{n,.n})$ changes with $n$, even it refers to a single true function $f$. 
The discussion in Section~\ref{SectionAliasing} shows that this difference is not trivial.

Likelihood-based empirical Bayes and hierarchical Bayes estimation of the scale parameter $c$
in the infinite sequence model were studied in \cite{Szabo2}. Besides considering the finite sequence model,
in the present paper we also study the risk-based empirical Bayes method and allow more
general priors. A main difference is that we have focused on the coverage of the credible sets. 
Such coverage is also studied in \cite{Szabo}, but only for the likelihood-based empirical Bayes method
in the infinite-sequence model with $\mathcal{N}(0, i^{-1-2\a})$-priors and $\a$ taken equal to the smoothing parameter.
The focus in the present paper on balls in the space of the finite vectors $\vec{f}_n$ of function values allows us to make the
connection to the correctness of a fraction of the credible intervals, as in Corollary~\ref{CorollaryIntervals}.
The present paper also differs in its technical details and proofs, in that our results
are directly formulated in terms of the criterion that is optimized, whereas \cite{Szabo,Szabo2} 
make the derivative of the criterion intercede. The present approach gives better insight and allows to
state the contribution of the (discrete) polished tail condition more precisely, with the 
possibility of generalisation to the good bias condition (\ref{eq:EqGoodBias}), which is
dependent both on the method and the prior.

Throughout, we limit the estimator to the interval $\interval$. This is reasonable, 
since the optimal rate of rescaling for functions in a class of smoothness $\a$ satisfies $cn\asymp n^\d$, where $\d=m/(1+2\a)\in(0,m]$
(if $\a\in(0, m)$ or $\a\in (0,m/2)$ in the risk-based and likelihood-based methods).

We consider the hierarchical Bayes  only with the usual inverse Gamma prior on the scaling parameter.
From the proof it is not difficult to see that the result extends to more general priors. For instance
if $c^{-r}\sim \Gamma(\kappa,\lambda)$, for some $r>0$, then the theorem is again true, but with the
term $1/c$ replaced by $(1/c)^r$. A choice $r\le 1$ does not change much, but the choice $r>1$
has an adverse effect on the rate of contraction for Sobolev classes: optimality is obtained
only for $\a\le (1/r+m-1)/2$.

The assumption that the errors in the regression model are normally distributed is crucial
to define the posterior distribution and credible sets. However, the derivation of the properties of these
objects uses only that the errors have mean zero and finite fourth moments. Thus the standard normal
model may be misspecified. This is true in particular regarding the assumption of unit variance, although
it would be preferable to extend our results to allow for a prior on this variance.

The study of credible bands, rather than credible balls or credible intervals in a fractional sense, would require
control of the bias of the posterior mean in a uniform sense. This involves properties of the
eigenvectors of the priors and goes beyond the ``$\ell_2$-theory'' considered in the present paper.
The bias in the example of Brownian motion is considered in detail in \cite{Sniekers}. We hope
to employ this in the study of credible bands in future work.

\section{Acknowledgements}
We thank Johannes Schmidt-Hieber for pointing out the special formulas for Brownian motion and its
primitives. 

We realised the connection between Brownian motion on its special grid and the discrete one-dimensional
Laplacian after hearing a presentation by Alice Kirichenko of her joint work with Harry van Zanten (on the Laplacian on general graphs).

\section{Technical proofs}
\label{SectionProofs}
In this section we give the proofs of Corollary~\ref{CorollaryIntervals} and Propositions~\ref{prop:R},~\ref{prop:R4} and~\ref{prop:postunif}.

\subsection{Proof of Corollary~\ref{CorollaryIntervals}}
In the Bayesian model (\ref{EqBayesianModel}) we have $\vec Y_n=\vec W_n^c+\vec\e_n$ for independent vectors
$\vec W_n^c$ and $\vec \e_n$. The marginal posterior distribution of $f(x)$ given $c$ and $\vec Y_n$ is
the conditional law of $W_x^c$ given $c$ and $\vec Y_n$. By the assumed Gaussianity, this is
a normal law with mean the conditional expectation $\hat f_{n,c}(x)=\E (W_x^c\given \vec Y_n,c)$ and variance
equal to 
$$s_n^2(c,x)=\var\big[ W_x^c\given c,\vec Y_n\bigr]
=\var\bigl[W_x^c-\E (W_x^c\given \vec Y_n,c)\given c\bigr]
=\inf_a \E\bigl[(W_x^c-a^T \vec Y_n)^2\given c\bigr].$$
%When evaluated at a design point $x=x_{i,n}$ this is the expression in (\ref{EqUniformDesignPoints}), which
%is up to constants independent of $i$ by assumption. Alternatively, the posterior variance 
%at $x=x_{i,n}$ is equal to the $i$th diagonal element of the posterior covariance matrix $I-\Sigma_{n,c}^{-1}$.
%Hence the sum of the posterior variances over the design points is the trace of this matrix. We conclude that
%uniformly in $i\in \{1,\ldots, n\}$ it holds that
%$$s_n^2(c,x_{i,n})\asymp \frac 1n\tr(I-\Sigma_{n,c}^{-1})=\frac{s_n^2(c)}{n},$$
%where $s_n^2(c)$ is given in (\ref{EqPosteriorVariance}). It follows that the radius $Mr_n(c,x_{i,n})$ of the
%empirical Bayes interval $\hat C_{n,\eta,M}(x_{i,n})$ is asymptotic to a universal multiple of $Mz_\eta s_n(c)/\sqrt n$.
When evaluated at a design point $x=x_{i,n}$, this is equal to the $i^{\text{th}}$ diagonal element of the posterior covariance matrix $I-\Sigma_{n,c}^{-1}$. Hence the sum of the posterior variances over the design points is the trace of this matrix. It follows that for all $i\in J_n$ we have
$$s_n^2(c,x_{i,n})\gtrsim \frac 1n\tr(I-\Sigma_{n,c}^{-1})=\frac{s_n^2(c)}{n},$$
where $s_n^2(c)$ is given in (\ref{EqPosteriorVariance}). It follows that for $i\in J_n$ the radius $Mr_n(c,x_{i,n})$ of the empirical Bayes interval $\hat C_{n,\eta,M}(x_{i,n})$ is bounded from below (up to a universal multiple) of $Mz_\eta s_n(c)/\sqrt n$.

The function $f$ fails to belong to the empirical Bayes interval $\hat C_{n,\eta,M}(x)$
if and only if $|f(x)-\hat f_{n,\hat c_n}(x)| \ge M r_n(\hat c_n,\eta,x)$. Therefore, by Markov's inequality
$$\frac1n \sum_{i\in J_n} 1\bigl\{f \notin \hat{C}_{n,\eta,M}(x_{i,n})\bigr\}
\le \frac 1n\sum_{i\in J_n} \frac{|f(x_{i,n})-\hat f_{n,\hat c_n}(x_{i,n})|^2}{M^2 r_n^2(\hat c_n,\eta,x_{i,n})}
\lesssim  \frac{\|\vec f_n-\hat f_{n,\hat c_n}\|^2}{M^2 z_\eta^2 s_n^2(\hat c_n)}.$$
As noted in the first paragraph of the proof of Theorem~\ref{TheoremCoverageEB},
$s_n^2(\hat c_n)$ is asymptotic to the square radius $r_n^2(\hat c_n,\eta')$ of the credible balls
of the form (\ref{EqCredibleSetEB}), for any $\eta'\in (0,1)$. 
Therefore, if the left-hand is bigger than $1-\gamma$, then $f\notin \hat C_{n,M',\eta}$ for $M'$ a multiple of $Mz_\eta$. 
By Theorem~\ref{thm:honest} this is the case with probability tending to zero if $M'$ is
sufficiently large, which it is if $M$ is large. The result then follows, since
\[
\frac1n \sum_{i\in J_n} 1\bigl\{f \in \hat{C}_{n,\eta,M}(x_{i,n})\bigr\} + \frac1n \sum_{i\in J_n} 1\bigl\{f \notin \hat{C}_{n,\eta,M}(x_{i,n})\bigr\} = \frac{|J_n|}{n}\, \to\, 1.
\]

If the function $f$ fails to belong to 
the hierarchical interval $\hat C_{n,\eta,M}(x)$, then $|f(x)-\hat f_{n,\bar c_n}(x)| \ge M r_n(\bar c_n,\eta_2,x)$, 
for $\bar c_n$ as defined in the proof of Theorem~\ref{thm:hcoverage}. The rest of the proof is similar
to the proof of the empirical Bayes intervals.

The assertions concerning the radii are immediate from the corresponding assertions of 
Theorem~\ref{thm:honest} and the equivalences
$s_n(c,x_{i,n})\asymp s_n(c)/\sqrt{n}\asymp r_n(c,\eta)/\sqrt{n}$ uniformly for $i\in J_n$ under the extra assumption on the posterior variances.

\subsection{Proof of final assertion of Lemma~\ref{prop:D2nieuw}}
\label{SectionDetailsDoubleSum}
That $D_{2,n^2}^R$ and $s_{n^2}$ behave as claimed is immediate from Lemma~\ref{lem:dubsom};
we only need consider the behaviour of $D_{2,n^2}^L$. The derivative of this function is  given by
$c\mapsto c^{-1}D_{2,n^2}^R(c)$ and hence is asymptotic to $c^{-1}(cn^2)^{1/m}k_n(c)$ uniformly on
the interval $[l_n/n^2,n^{2m-2}]$, for any $l_n\ra\infty$. Here $k_n(c)=1+\log (cn^2)$ for $cn^2\le n^m$ and $k_n(c)=1+\log(n^{2m}/(cn^2))$ for $cn^2\ge n^m$.
Now, as $cn^2\ge l_n\ra\infty$, we have for $cn^2\le n^m$
$$\int_0^c s^{-1}(sn^2)^{1/m}k_n(s)\,ds=\int_0^{cn^2}u^{1/m-1}(1+\log u)\,du\asymp (cn^2)^{1/m}\log (cn^2),$$
since $\int_0^t u^{1/m-1} \log u\,du=mt^{1/m}\log t-m^2t^{1/m}$.
Furthermore, for $cn^2\in [n^m, n^{2m}]$ we have
\begin{align*}
\int_0^c s^{-1}(sn^2)^{1/m}k_n(s)\,ds&\asymp n\log n+\int_{n^m}^{cn^2}u^{1/m-1}\bigl(1+\log n^{2m}-\log u\bigr)\,du\\
&= n\log n+ m \bigl(1+\log (n^{2m}/u)\bigr) u^{1/m}\big|_{n^m}^{cn^2}+m\int_{n^m}^{cn^2}u^{1/m-1}\,du\\
&\asymp (cn^2)^{1/m}\bigl(1+\log (n^{2m}/cn^2)\bigr).
\end{align*}
Combining the two displays we see that in both cases the left side is asymptotic to $(cn^2)^{1/m}k_n(c)$.
This order does not change if we limit the integrals to the interval $[l_n/n^2,c]$, for $l_n\ra\infty$ slowly.
It follows that $D_{2,n^2}^L(c)$ has this order, provided the integral $\int_0^{l_n/n^2} (D_{2,n^2}^L)'(s)\,ds$ is 
of lower order. Since $(D_{2,n^2}^L)'(s)\lesssim \sumin\sumjn (ij)^{-2m} sn^4$, the latter integral is bounded
by a multiple of $l_n^2$, which is  of lower order again if $l_n\ra\infty$ sufficiently slowly.

\subsection{Proof of Proposition~\ref{prop:R}}
The proof is based on two lemmas. 

\begin{lem}
\label{lem:varinpunt}
\label{lem:incrR}
For the functions in both (\ref{EqDandRRisk}) and (\ref{EqDandRLik}) and any $c$ and $s<t$ in $(0,\infty)$ we have %\onee{in Lemma~\ref{lem:meervarinpunt} staat juist $\interval$, dat aanpassen?}
\begin{align*}
\var\bigl[ R_{1,n}(c,f) \bigr] &\lesssim D_{1,n}(c,f),\\
\var\bigl[ R_{2,n}(c)\bigr] &\lesssim D_{2,n}(c),\\
\var\bigl[R_{1,n}(s,f) - R_{1,n}(t,f) \bigr]& \lesssim \frac{(t-s)^2D_{1,n}(s,f)}{s^2},\\
\var\bigl[R_{2,n}(s) - R_{2,n}(t) \bigr]& \lesssim \frac{(t-s)^2D_{2,n}(s)}{s^2}.
\end{align*}
\end{lem}

\begin{proof} 
For the risk-based remainder $R_{1,n}^R$ given in (\ref{EqDandRRisk}) we have 
$$\var\bigl[ R_{1,n}^R(c,f) \bigr] = 4 \sum_{j=1}^n \frac{f_{j,n}^2}{(1+c\lambda_{j,n})^4} \leq 4 D_{1,n}^R(c,f). $$
The bound on the variance of the likelihood-based remainder $R_{1,n}^L$  in (\ref{EqDandRLik}) is very similar.
For $R_{2,n}^R$ in (\ref{EqDandRRisk}) we have 
\[
\var\bigl[R_{2,n}^R(c)\bigr] 
= 2\sum_{j=1}^n\frac{(2c\lambda_{j,n}+c^2\lambda_{j,n}^2)^2}{(1+c\lambda_{j,n})^4}
\le 8\sum_{j=1}^n\frac{(c\lambda_{j,n})^2}{(1+c\lambda_{j,n})^2}=8D_{2,n}^R(c).
\]
For the likelihood-based remainder in (\ref{EqDandRLik}) we have 
\[
\var\bigl[R_{2,n}^L(c)\bigr] 
= 2\sum_{j=1}^n\frac{(c\lambda_{j,n})^2}{(1+c\lambda_{j,n})^2}=2D_{2,n}^R(c)
\le 4D_{2,n}^L(c), 
\]
in view of the inequality $\log(1+x)-x/(1+x)\ge x^2/(1+x)^2/2$ for $x>0$.

The third and fourth assertions of the lemma follow by applying Lemma~\ref{lem:heelraar}.
For  the risk-based remainder given in (\ref{EqDandRRisk}), we use the lemma with the choices:
\begin{itemize} 
\item for $R_{1,n}^R$: $(\a,\b)=(0,2)$, $a_j=2f_{j,n}$, $U_j=Z_{j,n}$ and $(\delta,\gamma)=(0,2)$,
where the sum in \eqref{eq:raarcondition} becomes $4D_{1,n}^R$, 
\item  for $R_{2,n}^R$: $(\a,\b)=(0,2)$, $a_j=1$, $U_j=Z_{j,n}^2-1$ and $(\delta,\gamma)=(2,2)$,
where  the sum in \eqref{eq:raarcondition} becomes $D_{2,n}^R$.
\end{itemize}
For  the likelihood-based remainder, given in (\ref{EqDandRLik}), we use the lemma with the choices:
\begin{itemize} 
\item for $R_{1,n}^L$: $(\a,\b)=(0,1)$, $a_j=2f_{j,n}$, $U_j=Z_{j,n}$ and $(\delta,\gamma)=(0,1)$,
where the sum in \eqref{eq:raarcondition} becomes $4D_{1,n}^L$, 
\item  for $R_{2,n}^L$: $(\a,\b)=(1,0)$, $a_j=-1$, $U_j=Z{j,n}^2-1$ and $(\delta,\gamma)=(2,2)$,
where  the sum in \eqref{eq:raarcondition} will become $D_{2,n}^R$, which is bounded by a multiple of $D_{2,n}^L$.
\end{itemize}
This concludes the proof.
\end{proof}

\begin{lem}\label{lem:incrD}
For the functions in both (\ref{EqDandRRisk}) and (\ref{EqDandRLik}) and any $s< t$ in $\interval$ we have
\begin{align*}
\bigl|D_{1,n}(s,f)-D_{1,n}(t,f)\bigr| &\lesssim \frac{|t-s|D_{1,n}(s,f)}{s},\\
\bigl|D_{2,n}(s)-D_{2,n}(t)\bigr| &\lesssim \frac{|t-s|s_n^2(s)}{s}.
\end{align*}
\end{lem}

\begin{proof}
By Lemma~\ref{lem:raarplus} with $(\a,\beta)=(0,2)$ and $D^R_{1,n}$ as in (\ref{EqDandRRisk}) we have
\[
 |D_{1,n}^R(s,f)-D_{1,n}^R(t,f)| \leq \frac{|s-t|}{s} \sum_{j=1}^n \frac{f_{j,n}^2}{(1+s\lambda_{j,n})^2} = \frac{|s-t|}{s} D_{1,n}^R(s,f).
\]
The function $D_{1,n}^L$ in (\ref{EqDandRLik}) can be treated similarly, with the choice $(\a,\beta)=(0,1)$. 

Applying Lemma~\ref{lem:raarplus} with $(\a,\beta)=(2,0)$ to $D_{2,n}^R(c)$, we find
\[
 |D_{2,n}^R(s)-D_{2,n}^R(t)| \leq \frac{|s-t|}{s} \sum_{j=1}^n \frac{s\lambda_{j,n}}{(1+s\lambda_{j,n})^2}
\le \sum_{j=1}^n \frac{s\lambda_{j,n}}{1+s\lambda_{j,n}}.
\]
The right side is $s_n^2(s)$, by definition (\ref{EqPosteriorVariance}).
Applying the mean value theorem to $D_{2,n}^L$ in (\ref{EqDandRLik}) we find for some $s\leq \xi \leq t$,
\begin{align*}
 |D_{2,n}^L(s)-D_{2,n}^L(t)| &\leq |s-t| \sum_{j=1}^n \frac{\xi\lambda_{j,n}^2}{(1+\xi\lambda_{j,n})^2} 
\leq |s-t| \sum_{j=1}^n \frac{\lambda_{j,n}}{1+\xi\lambda_{j,n}}\\
 &\leq \frac{|s-t|}{s}\sum_{j=1}^n \frac{s\lambda_{j,n}}{1+s\lambda_{j,n}}.
\end{align*}
This concludes the proof.
\end{proof}

\begin{proof}[Proof of Proposition~\ref {prop:R}]
Applying Lemmas~\ref{lem:varinpunt} and~\ref{lem:incrD}, we see that for any $s< t$ in $\interval$
\begin{align*}
&\var\left(\frac{R_{1,n}(s,f)}{D_n(s,f)}-\frac{R_{1,n}(t,f)}{D_n(t,f)}\right) \\
&\qquad\leq 2\var\left(\frac{R_{1,n}(s,f)-R_{1,n}(t,f)}{D_n(s,f)}\right)+ 2 \var\bigl[R_{1,n}(t,f)\bigr]\left(\frac{D_n(s,f)-D_n(t,f)}{D_n(s,f)D_n(t,f)}\right)^2\\
&\qquad\lesssim \frac{(t-s)^2}{s^2D_n(s,f)}+ \frac{(t-s)^2}{s^2D_n(t,f)}\frac{D_{1,n}^2(s,f)+s_n^4(s)}{D_n^2(s,f)}\\
&\qquad\lesssim \frac{(t-s)^2}{s^{2+1/m} n^{1/m}},
\end{align*}
since $D_n(s,f) \geq D_{2,n}(s) \gtrsim (sn)^{1/m}\asymp s_n^2(s)$ by Lemma~\ref{prop:D2nieuw}.
Similarly, applying Lemma~\ref{lem:varinpunt} we see that
$$\var\left(\frac{R_{1,n}(s,f)}{D_n(s,f)}\right) \lesssim \frac 1{D_n(s,f)}\lesssim \frac{1}{(sn)^{1/m}}$$
by Lemma~\ref{prop:D2nieuw}. The result for $R_{1,n}$ follows from the preceding two displays, by application of 
Lemma~\ref{lem:stuklemma}.
The assertion for $R_{2,n}$ is proved analogously, from the other parts of Lemmas~\ref{lem:varinpunt} and~\ref{lem:incrD}.
\end{proof}

\subsection{Proof of Proposition~\ref{prop:R4}}

In addition to Lemma~\ref{lem:incrD} we need the following lemma.

\begin{lem}
\label{lem:meervarinpunt}
\label{lem:meerincrR}
For any $c$ and any $s<t$ in $(0,\infty)$ we have
\begin{align*}
\var\bigl[ R_{3,n}(c,f)\bigr] &\le 4 D_{1,n}^R(c,f),\\
\var\bigl[ R_{4,n}(c)\bigr] &\le 2 D_{2,n}^R(c),\\
\var\bigl[R_{3,n}(s,f) - R_{4,n}(t,f) \bigr] &\lesssim \frac{(t-s)^2D_{1,n}^R(s)}{s^2},\\
\var\bigl[R_{4,n}(s) - R_{4,n}(t) \bigr] &\lesssim \frac{(t-s)^2D_{2,n}^R(s)}{s^2}.
\end{align*}
\end{lem}

\begin{proof} 
For the first two inequalities we compute 
\begin{align*}
\var\left[R_{3,n}(c,f)\right] &= 4 \sum_{j=1}^n \frac{(c\lambda_{j,n})^2 f_{j,n}^2}{(1+c\lambda_{j,n})^4} \le 4 D_{1,n}^R(c,f),\\
\var\left[R_{4,n}(c)\right] &= 2 \sum_{j=1}^n \frac{(c\lambda_{j,n})^4}{(1+c\lambda_{j,n})^4} \le 2 D_{2,n}^R(c).
\end{align*}
The third and fourth inequalities follow by application of Lemma~\ref{lem:heelraar} with the following choices:
\begin{itemize} 
\item for $R_{3,n}$: $(\a,\b)=(1,1)$, $a_j=-2f_{j,n}$, $U_j=Z_{j,n}$ and $(\delta,\gamma)=(0,2)$, where the sum in \eqref{eq:raarcondition} becomes $4D_{1,n}^R$.
\item for $R_{4,n}$: $(\a,\b)=(2,0)$, $a_j=1$, $U_j=Z_{j,n}^2-1$ and $(\delta,\gamma)=(2,2)$, where the sum in \eqref{eq:raarcondition} becomes $D_{2,n}^R$.
\end{itemize}
This concludes the proof.
\end{proof}

\begin{proof}[Proof of Proposition~\ref{prop:R4}]
Using Lemmas~\ref{lem:meerincrR} and~\ref{lem:incrD}, we have for $s<t$ in $\interval$
\begin{align*}
&\var\left(\frac{R_{3,n}(s,f)}{D_n^R(s,f)}-\frac{R_{3,n}(t,f)}{D_n^R(t,f)}\right) \\
&\qquad\leq 2\var\left(\frac{R_{3,n}(s,f)-R_{3,n}(t,f)}{D_n^R(s,f)}\right)
+ 2 \var\bigl[R_{3,n}(t,f)\bigr]\left(\frac{D_n^R(s,f)-D_n^R(t,f)}{D_n^R(s,f)D_n^R(t,f)}\right)^2\\
&\qquad\lesssim \frac{(t-s)^2}{s^2D_n^R(s,f)}+ \frac{(t-s)^2}{s^2D_n^R(t,f)}\frac{D_{1,n}^R(s,f)^2+s_n^4(s)}{D_n^R(s,f)^2}\\
&\qquad\lesssim \frac{(t-s)^2}{s^{2+1/m} n^{1/m}},
\end{align*}
since $D_{1,n}^R\leq D_n^R$ and $D_n^R(t,f) \geq D_{2,n}^R(t) \gtrsim (sn)^{1/m}\asymp s_n^2(s)$ by Lemma~\ref{prop:D2nieuw}.
Similarly, we have by Lemma~\ref{lem:meervarinpunt}
$$\var\left(\frac{R_{3,n}(s,f)}{D_n^R(s,f)}\right) \le \frac 1{D_n^R(s,f)}\lesssim \frac{1}{(sn)^{1/m}},$$
by Lemma~\ref{prop:D2nieuw}. The proposition with $D_n=D_n^R$ follows by an application of 
Lemma~\ref{lem:stuklemma}. 

%Using Lemmas~\ref{lem:meerincrR} and~\ref{lem:incrD}, we have, for $s<t$ in $\interval$, \onee{op $[l_n/n,n/l_n]$}
%\begin{align*}
%&\var\left(\frac{R_{4,n}(s,f)}{D_n^R(s,f)}-\frac{R_{4,n}(t,f)}{D_n^R(t,f)}\right) \\
%&\qquad\leq 2\var\left(\frac{R_{4,n}(s,f)-R_{4,n}(t,f)}{D_n^R(s,f)}\right)
%+ 2 \var\bigl[R_{4,n}(t,f)\bigr]\left(\frac{D_n^R(s,f)-D_n^R(t,f)}{D_n^R(s,f)D_n^R(t,f)}\right)^2\\
%&\qquad\lesssim \frac{(t-s)^2}{s^2D_n^R(s,f)}+ \frac{(t-s)^2}{s^2D_n^R(t,f)}\\
%&\qquad\lesssim \frac{(t-s)^2}{s^{2+1/m} n^{1/m}},
%\end{align*} \onee{$R_{4,n}$ heeft geen argument $f$}
%since $D_n^R(t,f) \geq D_{2,n}^R(t) \gtrsim (sn)^{1/m}$, by Lemma~\ref{prop:D2nieuw}.
%Similarly, by Lemma~\ref{lem:meervarinpunt},
%$$\var\left(\frac{R_{4,n}(s,f)}{D_n^R(s,f)}\right) \le \frac 1{D_n^R(s,f)^2}\lesssim \frac{1}{(sn)^{1/m}},$$
%by Lemma~\ref{prop:D2nieuw}. The proposition with $D_n=D_n^R$ follows by an application of 
%Lemma~\ref{lem:stuklemma}. 

Since $D_n^L\ge D_n^R/2$, this immediately implies the proposition for the likelihood-based
norming. The assertion for $R_{4,n}$ is proved analogously, from the other parts of Lemmas~\ref{lem:meerincrR} and~\ref{lem:incrD}.
\end{proof}

\subsection{Proof of Proposition~\ref{prop:postunif}}

\begin{lem}
\label{lem:sigman}
For $s\leq t$ we have
 \[
 \bigl|s_n^2(t) - s_n^2(s)\bigr| \lesssim \frac{|t-s|s_n^2(s)}{s}.
 \]
\end{lem}

\begin{proof}
This is immediate from the definition of $s_n^2$ in (\ref{EqPosteriorVariance}) and 
Lemma~\ref{lem:raarplus} with $(\a,\b)=(1,0)$.
\end{proof}

\begin{proof}[Proof of Proposition~\ref{prop:postunif}]
It is immediate from the definition of $N_n$ that
\[
\E\left[\frac{N_n(c)}{s_n^2(c)}-1\right] = 0, \qquad \var\left[N_n(c)\right] \lesssim s_n^2(c).
\]
Applying Lemma~\ref{lem:heelraar} with $(\a,\b)=(1,0)$, $a_j=1$, $(\gamma,\delta)=(1,1)$ and $g=s_n^2$, we find that for $s\leq t$
\[
\var\left[N_n(s)-N_n(t)\right] \lesssim \frac{(t-s)^2s_n^2(s)}{s^2}.
\]
It follows that
\begin{align*}
\var\left(\frac{N_n(s)}{s_n^2(s)}-\frac{N_n(t)}{s_n^2(t)}\right) 
&\le 2\var\left(\frac{N_n(s)-N_n(t)}{s_n^2(s)}\right)+ 2 \var\bigl[N_n(t)\bigr]\left(\frac{s_n^2(s)-s_n^2(t)}{s_n^2(s)s_n^2(t)}\right)^2\\
&\lesssim \frac{(t-s)^2 }{s^2s_n^2(s)}+\frac{(t-s)^2 }{s^2s_n^2(t)}\\
& \lesssim \frac{(t-s)^2}{s^{2+1/m} n^{1/m}},
\end{align*}
by Lemma~\ref{prop:D2nieuw}. The proposition follows by an application of Lemma~\ref{lem:stuklemma}. 
\end{proof}

For Brownian motion, we can gain more insight in the behaviour of (part of) the function $D_2^L$.
\begin{lem}\label{lem:det}
For the Brownian motion prior and $c\in[\log n/n,n]$,
\[
\log \det\Sigma_{n,c} \sim \sqrt{cn}.
\] 
\end{lem}

\begin{proof}
We want to find the determinant of the $n\times n$ matrix
\[
\Sigma=c\begin{pmatrix}\tfrac{1}{c}+\tfrac{1}{n_+} & \tfrac{1}{n_+} & \tfrac{1}{n_+} & \cdots & \tfrac{1}{n_+}\\[3pt]
 \tfrac{1}{n_+} & \tfrac{1}{c}+\tfrac{2}{n_+} & \tfrac{2}{n_+} & \cdots & \tfrac{2}{n_+}\\ \tfrac{1}{n_+} & \tfrac{2}{n_+} & \ddots &  & \vdots \\[3pt]
  \vdots & \vdots &
 &\tfrac{1}{c}+\tfrac{n-1}{n_+} & \tfrac{n-1}{n_+} \\[3pt]
 \tfrac{1}{n_+} & \tfrac{2}{n_+} & \cdots & \tfrac{n-1}{n_+} & \tfrac{1}{c}+\tfrac{n}{n_+}  \end{pmatrix} \sim \begin{pmatrix}
 2+\tfrac{c}{n_+} & -1 & 0 & \cdots & 0\\[3pt]
 -1 & 2+\tfrac{c}{n_+} & -1 & \cdots & 0\\
 0 & -1 & \ddots &  & \vdots \\[3pt]
 \vdots & \vdots & &2+\tfrac{c}{n_+} & -1 \\[3pt]
 0 & 0 & \cdots & -1 & 1+\tfrac{c}{n_+}
\end{pmatrix}.
\]
If we denote this determinant by $d_n$, we see that
\[
d_n=\left(2+\frac{c}{n_+}\right)d_{n-1} -d_{n-2},
\]
with $d_1=1+\frac{c}{n_+}$ and $d_2=\left(2+\frac{c}{n_+}\right)\left(1+\frac{c}{n_+}\right)-1$. Note that this is the same recurrence relation as~(2.2) in~\cite{Sniekers}. The solution is given by $d_n=A\lambda_+^n + B \lambda_-^n$, where
\[
 A=\frac{c^2+cn_+(3-\lambda_-)+n_+^2(1-\lambda_-)}{(\lambda_+-\lambda_-)\lambda_+n_+^2}, \qquad \lambda_{\pm} = 1+\frac{c}{2n_+} \pm \frac{\sqrt{c}}{2\sqrt{n_+}} \sqrt{4+\frac{c}{n_+}}.
\]
Note that $\lambda_+\lambda_-=1$. Since $\theta = \frac{c}{n_+}\to 0$ uniformly in $c\in\interval$, we have $\lambda_{\pm}\to 1$ and
\[
A= \frac{(1-\lambda_-)}{(\lambda_+-\lambda_-)} + o(1) = \frac{\frac{1}{2}\bigl(\sqrt{\theta(4+\theta)}-\theta\bigr)}{\sqrt{\theta(4+\theta)}}+o(1)\to\frac{1}{2}.
\]
It is easy to see that $B=\lambda_- A\sim A$. Furthermore, we have 
\[
 \log(\lambda_+^n) = n \left[\frac{\theta}{2} + \sqrt{\theta} \frac{\sqrt{4+\theta}}{2} - \frac{\theta}{2}\left(\frac{\sqrt{4+\theta}}{2}\right)^2 +O(\theta^{3/2}) \right]= n \sqrt{\theta} + O(n\theta^{3/2}).
\]
Finally, we have
\[
\log d_n -\log(A\lambda_+^n) = \log\left(1+\frac{B}{A}\lambda_-^{2n}\right)\to 0.
\]
The result follows.
\end{proof}

\section{Technical results}
\label{SectionTechnicalResults}

\begin{lem}
\label{LemmaCrossing}
Let $D_1: \interval \to (0,\infty)$ be a decreasing function and $D_2: \interval\to (0,\infty)$ an increasing function. Suppose that there exist $a,b,B,B'>0$ such that
\begin{align}
\label{EqPolishedTailConsequence}
D_1(Kc)&\le K^{-a}D_1(c),\qquad \text{ for any }\quad K>1,\\
\label{EqPropertyD2}
B' k^b D_2(c) \geq D_2(kc),&\ge Bk^b D_2(c)\qquad  \text{ for any }\quad  k<1.
\end{align}
Let $\tilde c$ satisfy $D_1(\tilde c)= D_2(\tilde c)$, and 
for a given constant $E\ge 1$, define $\Lambda=\bigl\{c: (D_1+D_2)(c)\le E\, (D_1+D_2)(\tilde c)\bigr\}$. Then
\begin{enumerate}
\item[(i)] $D_1(c)\le B^{-1}(2E)^{1+b/a} D_2(c)$, for every $c\in\Lambda$.
\item[(ii)] $\Lambda\subset \bigl[(2E)^{-1/a}\tilde c, (2EB')^{1/b}\tilde c\bigr]$.
\end{enumerate}
\end{lem}

\begin{proof}
(i). If $c\ge \tilde c$, then $D_1(c)\le D_2(c)$, since $D_1$ and $D_2$ are equal at $\tilde c$ and decreasing and
increasing respectively. The inequality in (i) is then satisfied, since  $B^{-1}(2E)^{1+b/a}\ge 1$. 
If $c<\tilde c$, then by \eqref{EqPolishedTailConsequence} with $K=\tilde c/c$ we have
$$(\tilde c/c)^aD_1(\tilde c)\le D_1(c).$$ 
If $c\in\Lambda$, then also
$$D_1(c)\le (D_1+D_2)(c)\le E(D_1+D_2)(\tilde c)=2E D_1(\tilde c)$$ by the definition of $\tilde c$. 
Concatenating these inequalities, we conclude that $(\tilde c/c)^a\le 2E$, or 
$c\ge b_1\tilde c$ for $b_1=(2E)^{-1/a}<1$.
Then, by monotonicity and \eqref{EqPropertyD2},
$$D_2(c)\ge D_2(b_1\tilde c)\ge B b_1^b D_2(\tilde c).$$
This is equal to $B b_1^b D_1(\tilde c)\ge B b_1^b/(2E)D_1(c)$ by the second last last display. 
This concludes the proof of (i).

(ii). The lower bound on $\Lambda$ in (ii) is equivalent to the inequality $c\ge b_1\tilde c$,
which was already obtained in the preceding proof of (i). For the upper bound
we first note that for every $c\in\Lambda$ we have $D_2(c)\le D_1(c)+D_2(c)\le E(D_1+D_2)(\tilde c)=
2E D_2(\tilde c)$, by the definition of $\tilde c$. If  $c>\tilde c$, then \eqref{EqPropertyD2} gives that
the right hand side is bounded above by $2E B' (\tilde c/c)^bD_2(c)$. Concatenation of the inequalities gives
that $1\le 2E B'(\tilde c/c)^b$.
\end{proof}

The following 
lemma is applied throughout to handle the sums that occur in both the deterministic and stochastic terms of $L$. 

\begin{lem}\label{lem:som}
Let $\gamma>-1$, $m\ge 1$ and $\nu \in\mathbb{R}$ such that $\gamma-m\nu < -1$. Then
\begin{equation}\label{eq:lemmading}
\sum_{j=1}^n \frac{j^\gamma}{(j^m+c n)^\nu} = C_{\gamma,\nu,m} (cn)^{\gamma/m-\nu+1/m}\bigl(1+o(1)\bigr)
\end{equation}
uniformly for $c\in [l_n/n,n^{m-1}/l_n]$ as $n\ra\infty$, for any $l_n\ra\infty$. The constant is given by
\[
C_{\gamma,\nu,m} = \int_0^\infty \frac{u^\gamma}{(u^m+1)^\nu} \intd u.
\]
Furthermore, the left side of (\ref{eq:lemmading}) has the same
order as the right side uniformly in $c\in [l_n/n,n^{m-1}]$ , for any $l_n\ra\infty$, possibly with a smaller constant.
\end{lem}

\begin{proof}
If $\gamma\le 0$, then the function $t\mapsto g(t)={t^\gamma}/{(t^m+c n)^\nu}$ is decreasing on
$[0,\infty)$, while if $\gamma>0$ the function is unimodal with a maximum at $k(c n)^{1/m}$ for 
the constant $k=(\gamma/(m\nu-\gamma))^{1/m}$.
In the first case we have 
\[
\int_1^n \frac{t^\gamma}{(t^m+c n)^\nu} \intd t 
\leq \sum_{j=1}^n \frac{j^\gamma}{(j^m+c n)^\nu}  
\leq \int_0^n \frac{t^\gamma}{(t^m+c n)^\nu} \intd t,
\]
while in the second case
\[
\int_1^n \frac{t^\gamma}{(t^m+c n)^\nu} \intd t - g(k (c n)^{1/m})  
\leq \sum_{j=1}^n \frac{j^\gamma}{(j^m+c n)^\nu}  
\leq \int_0^n \frac{t^\gamma}{(t^m+c n)^\nu} \intd t + g(k (c n)^{1/m}). 
\]
By the change of coordinates $t^m=(cn)u^m$ we have
\[
\int_a^n \frac{t^\gamma}{(t^m+c n)^\nu} \intd  t= (c n)^{\gamma/m -\nu+1/m} 
\int_{a/(cn)^{1/m}}^{n/(cn)^{1/m}} \frac{u^\gamma}{(u^m+1)^\nu} \intd  u.
\]
If $cn\ra\infty$ with $(cn)^{1/m}\ll n$, then for both $a=0$ and $a=1$ the integral 
on the right approaches $C_{\gamma,\nu,m}$, which is finite under the conditions of the lemma. 
The maximum value in the second display satisfies $g(k (c n)^{1/m}) \lesssim (cn)^{(\gamma/m - \nu)}$ and hence
is of lower order than the right side of the preceding display if $cn\ra\infty$. This proves the
first assertion of the lemma. For $c$ as in the second assertion we still have that
$cn\ra\infty$, so that the lower limit of the integral tends to zero, but the upper limit
$n/(cn)^{1/m}$ may remain bounded, although it is bigger than 1 by assumption.
\end{proof}

\begin{lem}\label{lem:dubsom}
For $\gamma>-1$, $m\ge 1$ and $\nu \in\mathbb{R}$ such that $\gamma-m\nu < -1$ we have
\begin{equation*}
\sum_{i=1}^n\sum_{j=1}^n \frac{(ij)^\gamma}{((ij)^m+c n^2)^\nu} \asymp (cn^2)^{\gamma/m-\nu+1/m}\times
\begin{cases}
\bigl(1+\log(cn^2)\bigr)&\text{ if }cn^2\le n^m,\\
\bigl(1+\log(n^{2m}/(cn^2))\bigr)&\text{ if }cn^2\ge n^m,
\end{cases}
\end{equation*}
uniformly for $c\in [l_n/n^2,n^{2m-2}]$ as $n\ra\infty$, for any $l_n\ra\infty$. 
\end{lem}

\begin{proof}
Since $cn^2\le (ij)^m+cn^2\le 2 cn^2$ if $ (ij)^m\le cn^2$ and $ (ij)^m\le (ij)^m+cn^2\le 2 (ij)^m$ otherwise,
the double sum is up to a constant $2^\nu$ bounded above and below by
$$\mathop{\sum_{i=1}^n\sum_{j=1}^n}_{(ij)^m\le cn^2} \frac{(ij)^\gamma}{(c n^2)^\nu}
+\mathop{\sum_{i=1}^n\sum_{j=1}^n}_{(ij)^m> cn^2} (ij)^{\gamma-m\nu}.$$
Since $cn^2\ge l_n\ra\infty$, the first sum is never empty; the second is empty if $cn^2=n^{2m}$
takes it maximally allowed value. To proceed we consider the cases that $N:= (cn^2)^{1/m}$ is smaller
or bigger than $n$ separately. If $N\le n$, then the second sum splits in two parts and the preceding display is 
equivalent to
\begin{align*}
&\sum_{i=1}^N\sum_{j=1}^{N/i} \frac{(ij)^\gamma}{N^{m \nu}}+ \sum_{i=1}^N\sum_{j=N/i+1}^n (ij)^{\gamma-m\nu}
+ \sum_{i=N+1}^n\sum_{j=1}^n (ij)^{\gamma-m\nu}\\
&\qquad \asymp
\sum_{i=1}^N \frac{i^\gamma(N/i)^{\gamma+1}}{N^{m \nu}}+ \sum_{i=1}^N i^{\gamma-m\nu}(N/i)^{\gamma-m\nu+1}+ \sum_{i=N+1}^ni^{\gamma-m\nu}\\
&\qquad \asymp
(\log N)N^{\gamma+1-m \nu}+ (\log N) N^{\gamma-m\nu+1}+ N^{\gamma-m\nu+1}.
\end{align*}
If $N>n$, then the first sum splits into two parts and we obtain the equivalent expression
\begin{align*}
&\sum_{i=1}^{N/n}\sum_{j=1}^{n} \frac{(ij)^\gamma}{N^{m \nu}}+ \sum_{i=N/n+1}^n\sum_{j=1}^{N/i} \frac{(ij)^\gamma}{N^{m \nu}}
+ \sum_{i=N/n+1}^n\sum_{j=N/i+1}^n (ij)^{\gamma-m\nu}\\
&\qquad \asymp
\sum_{i=1}^{N/n} \frac{i^\gamma n^{\gamma+1}}{N^{m \nu}}+ \sum_{i=N/n+1}^n \frac{i^\gamma(N/i)^{\gamma+1}}{N^{m \nu}}
+ \sum_{i=N/n+1}^ni^{\gamma-m\nu}(N/i)^{\gamma-m\nu+1}\\
&\qquad \asymp
N^{\gamma-m\nu+1}+ (\log (n^2/N))N^{\gamma-m \nu+1}+ (\log (n^2/N)) N^{\gamma+1-m\nu}.
\end{align*}
These bounds can be written in the form given by the lemma.
\end{proof}

%We take $\lambda_{i,j,n} \asymp \frac{n^2}{(i^2+j^2)^m}$, then  
%\[
%\sum_{i=1}^n\sum_{j=1}^n \left( \frac{c\lambda_{i,j,n}}{1+c\lambda_{i,j,n}}\right)^\nu =\sum_{i=1}^n\sum_{j=1}^n \left( \frac{cn^2}{(i^2+j^2)^m+c n^2}\right)^\nu
%\]
\begin{lem}\label{lem:dubsom2}
For $m\ge 1$ and $\nu \in\mathbb{R}$ such that $-m\nu < -1$, we have
\[
\sum_{i=1}^n \sum_{j=1}^n \frac{1}{\bigl((i^2+j^2)^m+c n^2\bigr)^\nu} \asymp (cn^2)^{-\nu+1/m}
\]
uniformly for $c\in [l_n/n^2,n^{2m-2}]$ as $n\ra\infty$, for any $l_n\ra\infty$. 
\end{lem}

\begin{proof}
Since the function $(s,t) \mapsto 1/\bigl((s^2+t^2)^m+c n^2\bigr)^\nu$ is decreasing in $s$ and $t$, we have
\[
 \sum_{i=1}^n \sum_{j=1}^n \frac{1}{\bigl((i^2+j^2)^m+c n^2\bigr)^\nu}\leq \int_0^n \int_0^n \frac{1}{\bigl((s^2+t^2)^m+c n^2\bigr)^\nu} \intd s \intd t
\]
and
\[
 \sum_{i=1}^n \sum_{j=1}^n \frac{1}{\bigl((i^2+j^2)^m+c n^2\bigr)^\nu} \geq \int_1^n \int_1^n \frac{1}{\bigl((s^2+t^2)^m+c n^2\bigr)^\nu} \intd s \intd t.
\]
Rewriting the double integrals in polar coordinates, we see that
\[
 \frac{\pi}{2} \int_{\sqrt{2}}^n \frac{r}{\bigl(r^{2m}+cn^2\bigr)^\nu} \intd r \leq \sum_{i=1}^n \sum_{j=1}^n \frac{1}{\bigl((i^2+j^2)^m+c n^2\bigr)^\nu}\leq \frac{\pi}{2} \int_0^{\sqrt{2}n} \frac{r}{\bigl(r^{2m}+cn^2\bigr)^\nu} \intd r.
\]
By the change of coordinates $r=\bigl(cn^2\bigr)^{\frac{1}{2m}} u$ we then have
\[
 \int_a^{bn} \frac{r}{\bigl(r^{2m}+cn^2\bigr)^\nu} \intd r =  \left(cn^2\right)^{-\nu+1/m}  \int_{a/(cn^2)^{1/(2m)}}^{bn/(cn^2)^{1/(2m)}} \frac{u}{\bigl(u^{2m}+1\bigr)^\nu} \intd u.
\]
Since $cn^2\to\infty$ the lower limit of this integral tends to zero. Combining this with the fact that the upper limit is bounded from below by $b$, the result follows.
\end{proof}

The following three lemmas are used to establish uniform bounds on the stochastic remainder terms.

\begin{lem}\label{lem:raarplus}
Consider a function $g: (0,\infty)\to \RR$ of the form
 \[
  g(c) = \frac{(c\lambda_{j,n})^\alpha}{(1+c\lambda_{j,n})^{\alpha+\beta}},
 \]
 where $\alpha,\beta \geq 0$ are integers.  Then, for $0<s<  t<\infty$,
 \[
  |g(s)-g(t)| \leq  \frac{|s-t|}{s} \frac{s\lambda_{j,n}}{(1+s\lambda_{j,n})^{2\vee(1+\beta)}}.
 \]
 In particular, if $\beta\geq 2$, then $|g(s)-g(t)|\leq \frac{|s-t|}{s} \frac{1}{(1+s\lambda_{j,n})^2}$.
\end{lem}
\begin{proof}
We apply the mean value theorem to the function $h(x) = {x^\alpha}/{(1+x)^{\alpha+\beta}}$. Note that for $x\geq 0$ we have  
\[\begin{split}
 |h'(x)| &= \left| \frac{x^{\alpha -1}\bigl(-\beta x +\alpha\bigr)}{(1+x)^{1+\alpha+\beta}}\right|\lesssim \frac{x^\alpha}{(1+x)^{1+\alpha+\beta}}1_{\beta\neq 0} + \frac{x^{\alpha -1}}{(1+x)^{1+\alpha+\beta}}\\
 &\leq \frac{1}{(1+x)^{1+\beta}}1_{\beta \neq 0} + \frac{1}{(1+x)^{2+\beta}} \lesssim \frac{1}{(1+x)^{2\vee(1+\beta)}}.
\end{split}\] 
Hence 
\[
  |g(s)-g(t)| \lesssim  |s-t| \frac{\lambda_{j,n}}{(1+s\lambda_{j,n})^{2\vee(1+\beta)}} = \frac{|s-t|}{s} \frac{s\lambda_{j,n}}{(1+s\lambda_{j,n})^{2\vee(1+\beta)}}.
\]
\end{proof}

\begin{lem}\label{lem:heelraar} 
Consider the stochastic process $(U(c): c>0)$ given by,
for constants $a_j$, i.i.d.~mean-zero random variables with finite variance  $U_j$ and integers $\alpha,\beta\geq 0$,
\[
U(c) = \sum_{j=1}^n \frac{a_j (c\lambda_{j,n})^\alpha}{(1+c\lambda_{j,n})^{\alpha+\beta}} U_j.
\]
Suppose that for some $\gamma,\delta \in \{0,1,2\}$ and some non-negative function $g$ we have
\begin{equation}\label{eq:raarcondition}
 \sum_{j=1}^n \frac{a_j^2 (s\lambda_{j,n})^\delta}{(1+s\lambda_{j,n})^\gamma}\lesssim g(s).
\end{equation}
Then, for  $0<s< t<\infty$,
\[
\var\bigl( U(s)-U(t)\bigr) \lesssim \frac{(s-t)^2 g(s)}{s^2}.
\]
\end{lem}

\begin{proof}
We consider
\[
 \var\bigl[ U(s)-U(t)\bigr] = \sum_{j=1}^n a_j^2 \left[\frac{(s\lambda_{j,n})^\alpha}{(1+s\lambda_{j,n})^{\alpha+\beta}} - \frac{(t\lambda_{j,n})^\alpha}{(1+t\lambda_{j,n})^{\alpha+\beta}}\right]^2.
\]
Applying the previous lemma, we see that 
\[
  \left|\frac{(s\lambda_{j,n})^\alpha}{(1+s\lambda_{j,n})^{\alpha+\beta}} - \frac{(t\lambda_{j,n})^\alpha}{(1+t\lambda_{j,n})^{\alpha+\beta}}\right| \lesssim \frac{|s-t|}{s}  \frac{s\lambda_{j,n}}{(1+s\lambda_{j,n})^2}.
\]
We conclude
\[\begin{split}
 \var\bigl[ U(s)-U(t)\bigr] \lesssim \frac{(s-t)^2}{s^2} \sum_{j=1}^n \frac{a_j^2 (s\lambda_{j,n})^2}{(1+s\lambda_{j,n})^4} \leq \frac{(s-t)^2}{s^2} \sum_{j=1}^n a_j^2 \frac{(s\lambda_{j,n})^\delta}{(1+s\lambda_{j,n})^\gamma},
\end{split}\]
which holds for any $\gamma,\delta \in \{0,1,2\}$. The result follows.
\end{proof}

\begin{lem}\label{lem:stuklemma}
Let $l_n\ra\infty$ be  a given sequence of numbers. 
If $U_n=(U_n(s): s\in I_n)$ are continuous stochastic processes such that for all $s<t$ in a closed interval $I_n\subset[l_n/n,\infty)$ and some $a>0$ we have
\begin{align*}
\E \bigl[U_n(s)\bigr]^2 &\lesssim \frac{1}{n^a s^{a}},
\qquad\qquad  \E\bigl[ U_n(s)-U_n(t)\bigr]^2 \lesssim \frac{(t-s)^2}{n^a s^{2+a}},
\end{align*}
then $\sup_{s\in I_n} | U_n(s)|$ tends to zero in probability.
\end{lem}

\begin{proof}
Write $I_n= [a_n,b_n]$. For a given interval $[s_0,t_0]\subset I_n$ we have $\E\big[ U_n(s)-U_n(t)\bigr]^2 \lesssim d_0^2(s,t)$, for
$d_0$ the metric 
$$d_0(s,t)=K_0|t-s|,\qquad K_0=n^{-a/2}s_0^{-1-a/2}.$$
The $d_0$-diameter of $[s_0,t_0]$ is $K_0|t_0-s_0|$ and the covering number 
$N(u,[s_0,t_0],d_0)$ is bounded above by $\bigl(K_0|t_0-s_0|/u\bigr)\vee 1$.
Therefore by Corollary~2.2.5 in \cite{Weak}, with $\psi(x)=x^2$, we have
\[
\E \sup_{s,t\in [s_0,t_0]}\big[ U_n(s)-U_n(t)\bigr]^2 \lesssim  K_0^2|t_0-s_0|^2=\frac{|t_0/s_0-1|^2}{(ns_0)^a}.
\]
Fix $M$ so that $2^{M-1}<1/a_n\le 2^M$
and $N$ so that $2^{N-1}<b_n\le 2^N$. Define $s_{-M}=a_n$, $s_N=b_n$ and $s_i=2^i$ for $i\in \{-M+1,\ldots,N-1\}$.
Then $s_{-M}<s_{-M+1}<\cdots<s_N$ partitions $I_n$. Since $s_{i+1}/s_i-1\le 1$ for every $i$ (in
fact, equal to $1$ except for the extremest values), we then have
\begin{align*}
\E \sup_{s\in \interval}U_n(s)^2
&\le2 \E \max_{i\in\{-M,\ldots,N-1\}} \biggl[\sup_{s\in [s_i,s_{i+1}]}|U_n(s)-U_n(s_i)|^2+U_n(s_i)^2\biggr]\\
&\lesssim \sum_{i=-M}^{N-1} \biggl[\frac{1^2}{(ns_i)^a}+\frac1{(ns_i)^a}\biggr]\\
&\lesssim \frac{1}{n^a}\sum_{i=-M}^{N-1}2^{-ia}=\frac{1}{n^a}2^{Ma}\frac{1-2^{-a(M+N)}}{1-2^{-a}}\\
&\le \frac{1}{n^a}\Bigl(\frac2{a_n}\Bigr)^a\frac{1}{1-2^{-a}} \le \frac{1}{l_n^a}\frac{2^{a}}{1-2^{-a}},
\end{align*}
by definition of $M$. This tends to zero, since $l_n\ra\infty$.
\end{proof}

\bibliographystyle{abbrv}
\bibliography{credibility}

\begin{thebibliography}{10}

\bibitem{Bull}
A.~Bull.
\newblock Honest adaptive confidence bands and self-similar functions.
\newblock {\em Electron. J. Statist.}, 6:1490--1516, 2012.

\bibitem{CaiLowMa}
T.~T. Cai, M.~Low, and Z.~Ma.
\newblock Adaptive confidence bands for nonparametric regression functions.
\newblock {\em J. Amer. Statist. Assoc.}, 109(507):1054--1070, 2014.

\bibitem{CaiLow}
T.~T. Cai and M.~G. Low.
\newblock An adaptation theory for nonparametric confidence intervals.
\newblock {\em Ann. Statist.}, 32(5):1805--1840, 2004.

\bibitem{CaiLow2006}
T.~T. Cai and M.~G. Low.
\newblock Adaptive confidence balls.
\newblock {\em Ann. Statist.}, 34(1):202--228, 2006.

\bibitem{Cox}
D.~D. Cox.
\newblock An analysis of bayesian inference for nonparametric regression.
\newblock {\em Ann. Statist.}, 21(2):903--923, 1993.

\bibitem{Freedman}
D.~Freedman.
\newblock On the {B}ernstein-von {M}ises theorem with infinite-dimensional
  parameters.
\newblock {\em Ann. Statist.}, 27(4):1119--1140, 1999.

\bibitem{GenWas}
C.~Genovese and L.~Wasserman.
\newblock Adaptive confidence bands.
\newblock {\em Ann. Statist.}, 36(2):875--905, 2008.

\bibitem{GGvdV}
S.~Ghosal, J.~K. Ghosh, and A.~W. van~der Vaart.
\newblock Convergence rates of posterior distributions.
\newblock {\em Ann. Statist.}, 28(2):500--531, 2000.

\bibitem{GineNickl}
E.~Giné and R.~Nickl.
\newblock Confidence bands in density estimation.
\newblock {\em Ann. Statist.}, 38(2):1122--1170, 2010.

\bibitem{Hoff}
M.~Hoffmann and R.~Nickl.
\newblock On adaptive inference and confidence bands.
\newblock {\em Ann. Statist.}, 39(5):2383--2409, 2011.

\bibitem{JudLam}
A.~Juditsky and S.~Lambert-Lacroix.
\newblock On nonparametric confidence set estimation.
\newblock {\em Math. Meth. of Stat}, 19(4):410--428, 2003.

\bibitem{KimeldorfWahba}
G.~Kimeldorf and G.~Wahba.
\newblock A correspondence between {B}ayesian estimation on stochastic
  processes and smoothing by splines.
\newblock {\em The Annals of Mathematical Statistics}, 41(2):495--502, 1970.

\bibitem{Bartek}
B.~Knapik, A.~W. van~der Vaart, and J.~H. van Zanten.
\newblock Bayesian inverse problems with gaussian priors.
\newblock {\em Ann. Statist.}, 39(5):2626--2657, 2011.

\bibitem{Low}
M.~G. Low.
\newblock On nonparamteric confidence intervals.
\newblock {\em Ann. Statist.}, 25(6):2547--2554, 1997.

\bibitem{RobVaart}
J.~Robins and A.~W. van~der Vaart.
\newblock Adaptive nonparametric confidence sets.
\newblock {\em Ann. Statist.}, 34(1):229--253, 2006.

\bibitem{Sniekers}
S.~Sniekers and A.~van~der Vaart.
\newblock Credible sets in the fixed design model with brownian motion prior.
\newblock {\em Journal of Statistical Planning and Inference}, (0):--, 2014.

\bibitem{Szabo2}
B.~T. Szabo, A.~W. van~der Vaart, and J.~H. van Zanten.
\newblock Empirical bayes scaling of gaussian priors in the white noise model.
\newblock {\em Electron. J. Statist.}, 7:991--1018, 2013.

\bibitem{Szabo}
B.~T. Szabo, A.~W. van~der Vaart, and J.~H. van Zanten.
\newblock {Frequentist coverage of adaptive nonparametric Bayesian credible
  sets}.
\newblock {\em to appear in Annals in Statistics}, 2015.

\bibitem{vdVvZ}
A.~van~der Vaart and H.~van Zanten.
\newblock Bayesian inference with rescaled {G}aussian process priors.
\newblock {\em Electron. J. Stat.}, 1:433--448 (electronic), 2007.

\bibitem{vdVvZGaussian}
A.~W. van~der Vaart and J.~H. van Zanten.
\newblock Rates of contraction of posterior distributions based on {G}aussian
  process priors.
\newblock {\em Ann. Statist.}, 36(3):1435--1463, 2008.

\bibitem{Weak}
A.~W. van~der Vaart and J.~A. Wellner.
\newblock {\em Weak convergence and empirical processes}.
\newblock Springer Series in Statistics. Springer-Verlag, New York, 1996.
\newblock With applications to statistics.

\bibitem{Wahba}
G.~Wahba.
\newblock Bayesian ``confidence intervals'' for the cross-validated smoothing
  spline.
\newblock {\em J. Roy. Statist. Soc. Ser. B}, 45(1):133--150, 1983.

\end{thebibliography}

\end{document}